\documentclass{amsart}

\usepackage{amsmath}
\usepackage{amssymb}
\usepackage{amsthm}
\usepackage{esint}
\usepackage[inline]{enumitem}
\usepackage{hyperref}
\hypersetup{hidelinks}

\theoremstyle{plain}
\newtheorem{theorem}{Theorem}[section]
\newtheorem{corollary}[theorem]{Corollary}
\newtheorem{lemma}[theorem]{Lemma}
\newtheorem{proposition}[theorem]{Proposition}

\theoremstyle{definition}
\newtheorem{definition}[theorem]{Definition}
\newtheorem{remark}[theorem]{Remark}

\theoremstyle{remark}

\numberwithin{equation}{section}

\newcommand{\R}{\mathbb{R}}

\newcommand{\trace}{\mathrm{tr}}
\newcommand{\dist}{\mathrm{dist}}
\newcommand{\diam}{\mathrm{diam}}

\newcommand{\cl}{\overline}
\newcommand{\loc}{\mathrm{loc}}

\DeclareMathOperator*{\divergence}{div}
\newcommand{\laplacian}{\Delta}

\DeclareMathOperator*{\osc}{osc}
\DeclareMathOperator*{\spt}{supp}
\newcommand{\capacity}{\mathrm{cap}}

\newcommand{\trinorm}[1]
{{
    \left\vert\kern-0.20ex\left\vert\kern-0.20ex\left\vert
    #1
    \right\vert\kern-0.20ex\right\vert\kern-0.20ex\right\vert
}}

\begin{document}

\title[Weighted and nonlocal Poisson equations]{Weighted and nonlocal Poisson equations on domains with capacity density conditions}
\author{Takanobu Hara}
\email{takanobu.hara.math@gmail.com}
\address{Mathematical Institute, Tohoku University,
6-3, Aramaki Aza-Aoba, Aoba-ku\\
Sendai, Miyagi 980-8578, Japan}
\date{20 October 2026}

\begin{abstract}
We establish necessary and sufficient conditions for global H\"older solvability of weighted Poisson equations and related fractional problems. On domains satisfying a capacity density condition, a nonnegative measure with bounded support depth admits a bounded globally H\"older continuous zero-boundary solution if and only if it satisfies an interior Morrey-type growth condition. Conversely, uniform boundary decay estimates for solutions with localized nonnegative sources characterize the capacity density condition. Analogous characterizations are established for the restricted and spectral fractional Laplacians. Our approach is based on weighted strong barriers and also yields two-weight Hardy inequalities, compact trace embeddings, and global H\"older solvability for nonlinear Dirichlet-to-Neumann operators.
\end{abstract}

\maketitle

\section{Introduction}
Let $\Omega \subset \R^{n}$, $n \ge 2$, be open with nonempty boundary, and let $1 < p < \infty$.
We consider
\begin{equation}\label{eqn:DE}
\begin{cases}
- \laplacian_{p, w} u := -\divergence(w |\nabla u|^{p - 2}\nabla u) = \nu & \text{in} \ \Omega, \\
u = 0 & \text{on} \ \partial \Omega,
\end{cases}
\end{equation}
where $w$ is a $p$-admissible weight (defined below),
and the signed Radon measure $\nu$ is locally finite.

Equation \eqref{eqn:DE} has been extensively studied in the literature, particularly with regard to the existence, uniqueness, and regularity of solutions, among many other aspects. Since our focus here is specifically on weighted equations, we confine ourselves to citing a few representative works,
including \cite{FKS1982} and \cite{MR2305115}.
For a survey of recent developments concerning weights that are singular
at the boundary, see \cite{MR4306765}.

In this paper, we establish global H\"{o}lder solvability for \eqref{eqn:DE}
and Hardy inequalities under capacity density conditions.
These conditions quantify the thickness of the complement without requiring
boundary smoothness. For the Poisson problem, the growth condition on the
forcing measure is imposed only on interior balls, allowing the measure to
accumulate near the boundary. We follow Ancona's strong-barrier approach
\cite{MR856511}, which was developed for weighted equations in
\cite{Hara2024} and for unweighted equations with measure data in
\cite{HaraHolder}. Motivated by applications to nonlocal problems, we develop
the H\"{o}lder theory for weighted equations with measure data and introduce
a two-weight formulation of the Hardy inequality.

We first apply the weighted theory to fractional Hardy inequalities with general measure weights under fractional capacity-density conditions. 
Fractional Hardy inequalities for zero-extended functions under geometric and capacity assumptions were studied in \cite{ILTVHardy,IMV2024}, and their relation to positive supersolutions was established in \cite{BianchiBrascoSkZagati}. 
Weighted trace estimates \cite{KoskelaSotoWang} allow us to obtain compact trace embeddings and quantitative lower-order inequalities, also on unbounded domains.

In the second application, we establish global H\"{o}lder solvability for measure-data Poisson problems involving the restricted and spectral fractional Laplacians and characterize the corresponding capacity-density conditions through estimates for localized sources. 
The classical theory of Riesz potentials, balayage, and Green kernels is presented in \cite{Landkof1972}, while the potential theory of stable processes and their boundary behavior is surveyed in \cite{BogdanEtAl2009}. 
To apply our weighted PDE methods to these nonlocal problems, we use the extension formulations \cite{CS2007,ST2010}.
The present paper was inspired primarily by \cite{Bjorn2024}, which investigates the Wiener criterion for the nonlocal Laplace equation. 
For the Wiener criterion and, of particular relevance to the present work,
the solvability of the boundary value problem with H\"{o}lder continuous boundary data, we refer the reader to \cite{Li2025,KimLeeWeidner2026} and the references therein.

We adopt the extension-based framework for nonlinear Dirichlet-to-Neumann
operators developed in \cite{CKKSS2025,GS2025}.
In this framework, the operator is defined through the boundary flux of a
$p$-harmonic extension, equivalently through the first variation of its energy.
The associated Neumann theory, including H\"older regularity for measure data,
has been further developed in \cite{CGKS2026}.
We apply this framework to the zero-exterior-value Dirichlet problem, using
the weighted Poisson theory established in Section 3.

The paper is organized as follows. Section 2 collects the weighted Sobolev
and measure-data estimates used below. Section 3 develops the weighted
Poisson theory under capacity density conditions and derives trace estimates,
compact embeddings, and Hardy inequalities. Section 4 applies these results
to fractional Hardy and trace inequalities. Section 5 treats Poisson problems
for the restricted and spectral fractional Laplacians, including uniqueness
and necessary capacity density conditions. Section 6 considers nonlinear
Dirichlet-to-Neumann operators defined through weighted $p$-harmonic extensions.

\section{Preliminaries and standard estimates}

We collect the weighted Sobolev and regularity estimates used in Section 3.
For the homogeneous theory we refer to \cite[Chapters 3, 6 and 20]{MR2305115};
see also \cite{MR757718}. The measure-data versions below follow from the
same iterations using the trace estimate in Lemma \ref{lem:trace}.

\subsection{Weights, Sobolev spaces, and capacity}
The letters $C,c>0$ denote constants that may change from line to line;
their dependence is specified in the relevant statements.
We write $A\asymp B$ if $C^{-1}B\le A\le CB$ for such a constant $C\ge1$.
Let $1<p<\infty$ and let $w\in L^1_{\loc}(\R^n)$ be positive almost everywhere.
Write $dw=w\,dx$, $w(E)=\int_E dw$, and, for a ball $B=B(x,R)$,
\[
\fint_B f\,dw:=\frac{1}{w(B)}\int_B f\,dw,
\qquad f_{B,w}:=\fint_B f\,dw,
\qquad aB:=B(x,aR).
\]
Here $a>0$.
The weight $w$ is \emph{$p$-admissible} if there are $C_D,C_P>0$ and
$\kappa\ge1$ such that, for every ball $B=B(x,R)$ and
$v\in C_c^\infty(\R^n)$,
\begin{gather*}
w(2B)\le C_Dw(B),\\
\fint_B|v-v_{B,w}|\,dw
\le C_PR\left(\fint_{\kappa B}|\nabla v|^p\,dw\right)^{1/p}.
\end{gather*}
Hereafter the structural constants are $n,p,C_D,C_P,\kappa$.
These assumptions imply, for some $\chi>1$, the Sobolev inequality
\begin{equation}\label{eqn:sobolev}
\left(\fint_B|v|^{p\chi}\,dw\right)^{1/(p\chi)}
\le CR\left(\fint_B|\nabla v|^p\,dw\right)^{1/p},
\end{equation}
for all $v\in C_c^\infty(B)$,
where $C$ and $\chi$ depend only on the structural constants.

Let $\Omega$ be an open set in $\R^{n}$. 
The weighted Sobolev space $H^{1,p}(\Omega;w)$ is the completion of
the functions in $C^\infty(\Omega)$ with finite norm
$(\int_\Omega(|v|^p+|\nabla v|^p)\,dw)^{1/p}$;
the gradient is the $L^p(w)$ limit of the gradients of approximating functions.
Its subspace $H^{1,p}_0(\Omega;w)$ is the closure of $C_c^\infty(\Omega)$.
Membership in $H^{1,p}_{\loc}(\Omega;w)$ means membership in
$H^{1,p}(G;w)$ for every open $G$ compactly contained in $\Omega$.
For $w=1$, we also use the notation $W^{1,p}$ for these spaces.
A weak solution of $- \laplacian_{p,w} u = \nu$ in $\Omega$ belongs to $H^{1,p}_{\loc}(\Omega;w)$ and satisfies
\[
\int_{\Omega} |\nabla u|^{p-2}\nabla u\cdot\nabla\varphi\,dw
=\int_{\Omega} \varphi \, d \nu,
\]
for all $\varphi\in C_c^\infty(\Omega)$.
Supersolutions and subsolutions use $\ge$ and $\le$, respectively, for nonnegative test functions. 
Sobolev boundary values are understood using quasicontinuous representatives \cite[Chapter 4]{MR2305115}.
In particular,
\[
\sup_{\partial \Omega}u := \inf\{k\in\R:(u-k)_+\in H^{1,p}_0(\Omega;w)\},
\qquad \inf_{\partial \Omega}u := -\sup_{\partial \Omega}(-u).
\]
Suprema, infima, and oscillations in the estimates below are essential values with respect to $dw$; 
for continuous representatives they are ordinary pointwise values.

For a compact set $K \subset \Omega$, set
\[
\capacity_{p,w}(K, \Omega):=
\inf\left\{\int_{\Omega} |\nabla\varphi|^p\,dw:
\varphi\in C_{c}^{\infty}(\Omega),\ \varphi\ge1\text{ on }K\right\}.
\]
Let $\Gamma \subset \partial\Omega$ be nonempty and closed.
We say that the capacity density condition holds on $\Gamma$ if there exists $\gamma>0$ such that
\begin{equation}\label{eqn:cdc}
\frac{\capacity_{p,w}(\cl{B(\xi,R)}\setminus\Omega,B(\xi,2R))}
{\capacity_{p,w}(\cl{B(\xi,R)},B(\xi,2R))}
\ge \gamma
\end{equation}
for all $\xi \in \Gamma$ and $R > 0$.
When $\Gamma=\partial\Omega$, we simply say that $\Omega$ satisfies the capacity density condition or is uniformly $(p, w)$-fat.
We also define 
\[
\delta_{\Gamma}(x) := \dist(x, \Gamma) \quad \text{and} \quad \delta(x) := \dist(x, \partial \Omega).
\]

\subsection{Estimates for measure data}
We set
\[
0 \le \lambda < p,
\quad
\beta = \frac{p - \lambda}{p - 1}.
\]
For $0 < R_{0} \le \infty$, we consider the global Morrey condition
\begin{equation}\label{eqn:full}
|\nu|(B(x,r))\le M_{\mathrm g}r^{ - \lambda } w(B(x,r)),
\qquad x\in\R^n,\quad 0 < r \le  8 \kappa R_{0},
\end{equation}
where $M_{\mathrm g}$ is a nonnegative constant. 
If a nonnegative Radon measure $\nu$ satisfies this condition, then it is absolutely continuous with respect to relative $(p,w)$-capacity on bounded open sets,
and we can define the integral $\int u \, d \nu$ for nonnegative $(p, w)$-quasicontinuous functions.

The following Adams-type inequality gives a scale-dependent estimate under
the measure growth condition \eqref{eqn:full}; compare
\cite[Corollary 1.95]{MR1461542} for the unweighted case and
\cite[Section 1.4.1]{Mazya2011} for the classical background.

\begin{lemma}\label{lem:trace}
Let $\nu$ be a nonnegative Radon measure satisfying \eqref{eqn:full}.
Then
\begin{equation}\label{eqn:trace}
\int_{\R^n}|v|^p\,d \nu
\le C M_{\mathrm g} \left(
r^{p - \lambda}\int_{\R^{n}} |\nabla v|^p \,dw
+
r^{- \lambda} \int_{\R^{n}} |v|^p \,dw
\right)
\end{equation}
for all $v\in C_c^\infty(\R^n)$ and $0 < r \le R_{0}$, where $C$ is a constant depending only on the structural constants and $\lambda$.
\end{lemma}

\begin{proof}
For trace estimates under the upper codimension condition
$\nu(B(x,\rho))\le C\rho^{-\lambda}w(B(x,\rho))$, see
\cite[Theorem 1.3 and the proof of Theorem 4.4]{Maly2017}.
We give the short argument to record the dependence on $r$.
Telescope the weighted averages of $v$ at radii $2^{-j} r$.
By the Poincar\'{e} and doubling conditions and Fubini's theorem,
the $L^p( \nu )$ norm of the $j$-th difference is controlled by
\[
C M_{\mathrm g}^{1/p}(2^{-j} r)^{1 - \lambda / p} \| \nabla v \|_{L^p(w)}.
\]
The initial average is bounded by $C M_{\mathrm g}^{1/p} r^{- \lambda / p} \|v\|_{L^p(w)}$.
Since $1 - \lambda/p > 0$, summation proves \eqref{eqn:trace}.
\end{proof}

The next three lemmas follow from standard Moser iteration for measure
data; see \cite[Section 3.1]{MR1461542}.
For this method with degeneracy given by powers of strong $A_\infty$
weights, see \cite[Sections 3--5]{DiFazioZamboni2006}, where a weighted
Fefferman--Poincar\'e inequality is used in Moser iteration.
Here Lemma \ref{lem:trace} supplies the required estimate: choosing its
scale gives, for every $\varepsilon>0$,
\[
\int |v|^p\,d|\nu|
\le \varepsilon\int |\nabla v|^p\,dw
+C_\varepsilon\int |v|^p\,dw.
\]
Applied to cutoff powers of the shifted solution, this absorbs the measure
term in the energy estimate. The Sobolev inequality then gives the
positive-power iteration for upper bounds and the negative-power and
logarithmic estimates for the weak Harnack inequality.
On a ball of radius $R$, the shift has size
$M_{\mathrm g}^{1/(p-1)}R^\beta$.

\begin{lemma}\label{lem:maximum}
Assume that \eqref{eqn:full} holds, and let $\Omega \subset B(x_0, R_{0})$ be an open set.
Suppose $u \in H^{1,p}(\Omega;w)$ satisfies $-\laplacian_{p,w}u = \nu$ in $\Omega$.
Then we have
\begin{equation}\label{eqn:maximum}
\sup_\Omega u \le \sup_{\partial \Omega}u + C M_{\mathrm g}^{1/(p-1)} R_{0}^{\beta},
\qquad
\inf_{\Omega} u \ge \inf_{\partial \Omega} u - C M_{\mathrm g}^{1/(p-1)} R_{0}^{\beta}.
\end{equation}
\end{lemma}

\begin{lemma}\label{lem:weak-harnack}
Assume that \eqref{eqn:full} holds with $0 < R \le R_{0}$. 
Let $u\ge0$ satisfy
$-\laplacian_{p,w}u\ge- \nu$ in $B(x_0,2R)$.
Fix $0 < q < \chi(p -1)$, where $\chi > 1$ is the exponent in the Sobolev inequality \eqref{eqn:sobolev}.
Then
\begin{equation}\label{eqn:harnack}
\left(\fint_{B(x_0,R)}u^q\,dw\right)^{1/q}
\le C \left(\inf_{B(x_0,R)}u+M_{\mathrm g}^{1/(p-1)}R^\beta\right),
\end{equation}
where $C$ is a constant depending only on the structural constants, $\beta$ and $q$.
Combining this inequality with the Caccioppoli estimate also gives
\[
R^{p-1}\fint_{B(x_0,R)}|\nabla u|^{p-1}\,dw
\le C\left(\inf_{B(x_0,R)}u+M_{\mathrm g}^{1/(p-1)}R^\beta\right)^{p-1},
\]
where $C$ depends only on the structural constants and $\beta$.
\end{lemma}

\begin{lemma}\label{lem:interior}
Assume \eqref{eqn:full} and $0 < R \le R_{0}$. 
Then every weak solution to $-\laplacian_{p,w}u = \nu$ in $B(x_0,2R)$ has a locally H\"older
continuous representative.
Moreover,
\begin{equation}\label{eqn:interior}
\osc_{B(x_0,r)}u\le C \left( \frac{r}{R} \right)^{ \alpha_{ \mathrm{in} } } 
\left(\osc_{B(x_0,2R)}u+M_{\mathrm g}^{1/(p-1)}R^\beta\right)
\end{equation}
for all $0 < r \le R$, where $C$ and $0 < \alpha_{ \mathrm{in} } <\min\{\beta,1\}$ are positive constants
depending only on the structural constants and $\beta$.
\end{lemma}

\begin{lemma}\label{lem:boundary}
Assume \eqref{eqn:full} holds with $0 < R \le R_{0}$. 
Fix $\xi\in\partial\Omega$ and suppose \eqref{eqn:cdc} holds with $\Gamma = \{ \xi \}$.
Let $u\in H^{1,p}(\Omega\cap B(\xi,R);w)$ be bounded and satisfy
\[
\begin{aligned}
-\laplacian_{p,w}u\le \nu, \quad u \ge m \quad \text{in } \Omega \cap B(\xi, R),
\\
u = m \quad \text{on } \partial\Omega\cap B(\xi,R/2).
\end{aligned}
\]
Then we have
\begin{equation}\label{eqn:boundary}
\sup_{\Omega\cap B(\xi,r)}(u-m)
\le C \left( \frac{r}{R} \right)^{ \alpha_{ \mathrm{bd}} }\sup_{\Omega\cap B(\xi,R/2)}(u-m)
+CM_{\mathrm g}^{1/(p-1)} R^{\beta}
\end{equation}
for all $0 < r \le R / 8$, where $C$ and $0 < \alpha_{ \mathrm{bd} } <\min\{\beta,1\}$ are positive constants
depending only on the structural constants, $\beta$ and $\gamma$.
\end{lemma}

\begin{proof}
This is the standard capacitary boundary estimate in \cite{MR1461542};
see also \cite[Section 4]{Hara2024}.
The gradient estimate in Lemma \ref{lem:weak-harnack}, combined with
\eqref{eqn:cdc}, gives the boundary contraction with error
$CM_{\mathrm g}^{1/(p-1)}r^\beta$.
Iteration proves \eqref{eqn:boundary}.
\end{proof}

\section{The weighted local equation}

\subsection{Solvability and capacity density}

In this section let $\Gamma\subset\partial\Omega$ be nonempty and closed.
Let $\nu$ be a signed locally finite Radon measure on $\Omega$ satisfying
\begin{equation}\label{eqn:morrey}
|\nu|(B(x, r)\cap\Omega) \le M r^{- \lambda} w(B(x, r)), \qquad x \in \Omega, \quad 0 < r < \delta_{\Gamma}(x) / 2. 
\end{equation}
When $\Gamma\ne\partial\Omega$, these balls need not lie in $\Omega$;
the measure is evaluated on their intersections with $\Omega$.
We further assume 
\[
\|\delta_{\Gamma}\|_{L^{\infty}(\spt\nu)} < \infty.
\]
Here $\spt\nu$ means the support of $|\nu|$ in $\Omega$, and the norm denotes
the supremum of the continuous distance function on that support.
For the zero measure its value is zero, and we take the zero solution.
We first construct a supersolution. For $\Gamma=\partial\Omega$,
Theorem \ref{thm:local} will then give a bounded solution $u$ satisfying
the estimate below, with
$0<\alpha_{\mathrm{ss}}<\min\{\alpha_{\mathrm{bd}},\beta,1\}$:
\begin{equation}\label{eqn:main-bound}
|u(x)| \le C M^{1/(p-1)} \|\delta_{\Gamma}\|_{L^{\infty}(\spt\nu)}^{\beta - \alpha_{ \mathrm{ss} }} \min\{\delta_{\Gamma}(x), \|\delta_{\Gamma}\|_{L^{\infty}(\spt\nu)}\}^{\alpha_{ \mathrm{ss} }}.
\end{equation}

\begin{proposition}\label{prop:barrier}
Assume \eqref{eqn:cdc} on $\Gamma$.
If $\nu \ge 0$ satisfies \eqref{eqn:morrey} and $0 < \|\delta_{\Gamma}\|_{L^{\infty}(\spt\nu)} < \infty$,
there is a continuous $S \in H^{1, p}_{\loc}(\Omega; w)$ such that
\begin{equation}\label{eqn:barrier}
- \laplacian_{p, w} S \ge \nu
\end{equation}
\begin{equation}\label{eqn:barrier2}
\begin{aligned}
&
\frac{1}{C} M^{1/(p-1)} \|\delta_{\Gamma}\|_{L^{\infty}(\spt\nu)}^{\beta - \alpha_{ \mathrm{ss} }}
\min\{\delta_{\Gamma}, \|\delta_{\Gamma}\|_{L^{\infty}(\spt\nu)}\}^{\alpha_{ \mathrm{ss} }} \\
& \quad \le
S \le C M^{1/(p-1)} \|\delta_{\Gamma}\|_{L^{\infty}(\spt\nu)}^{\beta - \alpha_{ \mathrm{ss} }}
\min\{\delta_{\Gamma}, \|\delta_{\Gamma}\|_{L^{\infty}(\spt\nu)}\}^{\alpha_{ \mathrm{ss} }}
\end{aligned}
\end{equation}
in $\Omega$.
Here $C$ and $\alpha_{ \mathrm{ss} } < \min\{ \alpha_{ \mathrm{bd} }, \beta, 1\}$
are positive constants depending only on the structural constants, $\beta$ and $\gamma$.
\end{proposition}

We record the minimum principle used to join the local supersolutions.
\begin{lemma}\label{lem:gluing}
Let $\mu$ be a signed locally finite Radon measure on $\Omega$.
If $v_1,v_2\in H^{1,p}_{\loc}(\Omega;w)\cap C(\Omega)$ satisfy
$-\laplacian_{p,w}v_i\ge\mu$ in $\Omega$ for $i=1,2$, then
\[
-\laplacian_{p,w}\min\{v_1,v_2\}\ge\mu\quad\text{in }\Omega.
\]
\end{lemma}
\begin{proof}
We use the truncation argument of \cite[Lemma 3.1]{Hara2024}.
Let $0\le\varphi\in C_c^\infty(\Omega)$, and let $\eta_a$ be a
nondecreasing Lipschitz function equal to zero on $(-\infty,-a]$
and to one on $[a,\infty)$.
Test the supersolution inequalities with
$\varphi\eta_a(v_2-v_1)$ and $\varphi(1-\eta_a(v_2-v_1))$,
respectively. The measure terms sum to $\int_\Omega\varphi\,d\mu$,
and the term containing $\eta_a'$ is nonpositive by monotonicity.
Discarding this term and letting $a\downarrow0$ gives the required
supersolution inequality.
\end{proof}

\begin{proof}[Proof of Proposition \ref{prop:barrier}]
Choose $0 < \theta < 1/16$ so that
$\theta^{\alpha_{ \mathrm{ss} }} < 1/8$ and
$C(2\theta)^{\alpha_{ \mathrm{bd} }} < \theta^{\alpha_{ \mathrm{ss} }}/8$.

Put $R_{0} := 2\|\delta_{\Gamma}\|_{L^{\infty}(\spt\nu)}/\theta$. Under the changes of scale
$x = R_{0}x'$ and $S(x) = M^{1/(p-1)} R_{0}^{\beta}S'(x')/\epsilon$, the new weight
is $w'(x') := w(R_{0}x')$ and the new datum is
\[
\nu'(E) := R_{0}^{p - n}
\left(\frac{\epsilon}{M^{1/(p-1)} R_{0}^{\beta}}\right)^{p - 1}\nu(R_{0}E).
\]
Since $(p - 1)\beta = p - \lambda$, its Morrey growth coefficient is at most
$\epsilon^{p-1}$. Its support lies in $\{\dist(\cdot,\Gamma') \le \theta/2\}$.
Here $\Gamma':=R_0^{-1}\Gamma$. Omit the primes in this normalization.

For $B := B(\xi, R)$, $\xi \in \Gamma$, $R \le 1$, set
\[
\nu_{B} := \mathbf{1}_{\Omega\cap B \cap \{\delta_{\Gamma} > \theta^{2}R\}}\nu.
\]
This measure satisfies the ball bound in \eqref{eqn:full}, with
$\nu_B(B(x,r)\cap\Omega)$ on the left and growth coefficient
$C_{\theta}\epsilon^{p-1}$.
Indeed, if a ball of radius $r < \theta^{2}R/8$ meets its support, recenter
at a support point and enlarge the radius to $2r$. The enlarged ball is
admissible in \eqref{eqn:morrey}. For $\theta^{2}R/8 \le r \le 8\kappa R$,
cover the support in the ball by balls of radius $\theta^{2}R/16$ with
bounded overlap. Summation and doubling give
\[
\nu_{B}(B(x, r)\cap\Omega) \le C \epsilon^{p - 1}(\theta^{2}R)^{- \lambda} w(B(x, C r)) \le C_{\theta} \epsilon^{p - 1} r^{- \lambda} w(B(x, r)).
\]
The upper range $8\kappa R$ is the one in \eqref{eqn:full}.

Choose smooth boundary values $\eta_{B}$ between
$(1 / 4) (\theta R)^{\alpha_{ \mathrm{ss} }}$ and $R^{\alpha_{ \mathrm{ss} }}$, equal to the former in
$(1 / 2) B$ and to the latter near $\Omega \setminus B$.
By \eqref{eqn:trace}, $\nu_B\in(H^{1,p}_0(\Omega\cap B;w))^*$.
The variational problem has a unique solution
$u_{B} - \eta_{B} \in H^{1, p}_{0}(\Omega \cap B; w)$.
Comparison and \eqref{eqn:maximum} give
\[
\frac{1}{4} (\theta R)^{\alpha_{ \mathrm{ss} }} \le u_{B} \le R^{\alpha_{ \mathrm{ss} }} + C_{\theta}\epsilon R^{\beta}.
\]
Choose $\epsilon$ small enough that this is at most $2R^{\alpha_{ \mathrm{ss} }}$.
Since $R^{\beta} \le R^{\alpha_{ \mathrm{ss} }}$, \eqref{eqn:boundary} gives, after
one further decrease of $\epsilon$,
\begin{equation}\label{eqn:bound_of_u_b_1}
\frac{1}{4} (\theta R)^{\alpha_{ \mathrm{ss} }} \le u_{B} \le 2R^{\alpha_{ \mathrm{ss} }},
\end{equation}
\begin{equation}\label{eqn:bound_of_u_b_2}
\sup_{\Omega \cap B(\xi, 2\theta R)}u_{B}
\le \frac{1}{4} (\theta R)^{\alpha_{ \mathrm{ss} }} 
+ \frac{1}{8} (\theta R)^{\alpha_{ \mathrm{ss} }}
+ \frac{1}{8} (\theta R)^{\alpha_{ \mathrm{ss} }}
=
\frac{1}{2} (\theta R)^{\alpha_{ \mathrm{ss} }}.
\end{equation}
Interior regularity gives continuity in $\Omega\cap B$.
At spherical boundary points in $\Omega$, the exterior ball condition gives
continuity with boundary value $R^{\alpha_{\mathrm{ss}}}$.
Define
\[
\widehat u_B(x)=
\begin{cases}
\min\{u_B(x),R^{\alpha_{\mathrm{ss}}}\},&x\in\Omega\cap B,\\
R^{\alpha_{\mathrm{ss}}},&x\in\Omega\setminus B.
\end{cases}
\]
The matching Sobolev boundary values give
$\widehat u_B\in H^{1,p}_{\loc}(\Omega;w)\cap C(\Omega)$.
For $0\le\varphi\in C_c^\infty(\Omega)$, test the equation for $u_B$ with
$\varphi\min\{(R^{\alpha_{\mathrm{ss}}}-u_B)_+/t,1\}$.
This test has zero Sobolev values on the spherical boundary.
The term from differentiating the truncation is nonpositive, so letting
$t\downarrow0$ shows that $-\laplacian_{p,w}\widehat u_B\ge0$ in $\Omega$.
On the open set $\{\widehat u_B<R^{\alpha_{\mathrm{ss}}}\}$,
the function equals $u_B$ and retains its original equation.

For each $k\ge0$, choose $\{\xi_{k,j}:j\in J_k\}\subset\Gamma$ such that
the balls $B(\xi_{k,j},\theta^k)$ are locally finite and
\[
\{x\in\Omega:\delta_\Gamma(x)\le\theta^{k+1}\}
\subset\bigcup_{j\in J_k}B(\xi_{k,j},2\theta^{k+1}).
\]
Set $v_k=\min_{j\in J_k}\widehat u_{B(\xi_{k,j},\theta^k)}$.
Locally, this is a minimum of finitely many functions and the constant
$\theta^{k\alpha_{\mathrm{ss}}}$. Thus $v_k$ is continuous, belongs to
$H^{1,p}_{\loc}(\Omega;w)$, and is a supersolution of the homogeneous
equation by Lemma \ref{lem:gluing}. Moreover,
\begin{equation}\label{eqn:scales1}
\frac14\theta^{(k+1)\alpha_{\mathrm{ss}}}
\le v_k\le\theta^{k\alpha_{\mathrm{ss}}},
\end{equation}
\begin{equation}\label{eqn:scales2}
v_k\le\frac12\theta^{(k+1)\alpha_{\mathrm{ss}}},
\end{equation}
whenever $\delta_\Gamma\le\theta^{k+1}$.
At any point where $v_k<\theta^{k\alpha_{\mathrm{ss}}}$, all functions
attaining the minimum agree locally with their untruncated solutions.
The other functions can be omitted in a neighborhood of that point.
Lemma \ref{lem:gluing} therefore gives
\begin{equation}\label{eqn:scales3}
-\laplacian_{p,w}v_k\ge
\mathbf1_{\{\delta_\Gamma>\theta^{k+2}\}}\nu,
\end{equation}
in the open set $\{v_k<\theta^{k\alpha_{\mathrm{ss}}}\}$.

Define on all of $\Omega$
\[
S(x)=\min\bigl(\{1\}\cup
\{v_k(x):k\ge0,\ \delta_\Gamma(x)<\theta^k\}\bigr).
\]
In particular, $S=1$ for $\delta_\Gamma\ge1$.
For $k\ge1$ and $\theta^{k+1}<\delta_\Gamma<\theta^k$,
\eqref{eqn:scales1} and $\theta^{\alpha_{\mathrm{ss}}}<1/8$ exclude
all indices $i\le k-2$, so $S=\min\{v_{k-1},v_k\}$.
Here \eqref{eqn:scales2} gives
\[
v_{k-1}\le\tfrac12\theta^{k\alpha_{\mathrm{ss}}}
<\theta^{k\alpha_{\mathrm{ss}}}.
\]
Thus every function attaining this minimum lies strictly below its
truncation level. Both right-hand sides in \eqref{eqn:scales3} equal
$\nu$ in this annulus, and Lemma \ref{lem:gluing} gives
$-\laplacian_{p,w}S\ge\nu$ there.

If $\delta_\Gamma(x)=\theta^k$, $k\ge1$, then
$v_k(x)=\theta^{k\alpha_{\mathrm{ss}}}$, whereas
$v_{k-1}(x)\le\tfrac12\theta^{k\alpha_{\mathrm{ss}}}$.
Continuity and the strict separation from the remaining scales imply
$S=v_{k-1}$ in a neighborhood of $x$.
Equation \eqref{eqn:scales3} gives $-\laplacian_{p,w}S\ge\nu$ in that
neighborhood as well, including any mass on the level set.
On $\{\delta_\Gamma>\theta\}$ we have $S=v_0$, also across
$\delta_\Gamma=1$, since $v_0=1$ for $\delta_\Gamma\ge1$.
This is a supersolution of the homogeneous equation, and $\nu$ vanishes
there because $\spt\nu\subset\{\delta_\Gamma\le\theta/2\}$.
These local representations prove
$S\in H^{1,p}_{\loc}(\Omega;w)\cap C(\Omega)$ and \eqref{eqn:barrier}.

By construction, $C^{-1}\min\{\delta_{\Gamma}, 1\}^{\alpha_{ \mathrm{ss} }}\le S\le
C\min\{\delta_{\Gamma}, 1\}^{\alpha_{ \mathrm{ss} }}$.
Returning to the original variables proves \eqref{eqn:barrier2}.
\end{proof}

\begin{remark}
When the datum exponent $\beta$ coincides with the optimal homogeneous
boundary decay exponent $\alpha$, the endpoint $\alpha_{ \mathrm{ss} }=\beta=\alpha$ can fail.
Let
\[
\Omega=\{(r, \phi) :0 < r < 1 ,\ - \omega < \phi < \omega\},
\qquad
\frac{\pi}{2} < \omega < \pi,
\]
and let $\alpha = \pi / (2 \omega)$.
The optimal boundary decay exponent at the origin is $\alpha$, 
as can be seen from the harmonic function $r^{\alpha} \cos( \alpha \phi)$, which vanishes on the two sides of the sector.
The function
\[
u(r,\phi)
=
\frac{1}{2 \alpha} r^{\alpha} \log \frac{1}{r} \cos(\alpha \phi)
\]
satisfies
\[
-\Delta u
= r^{\alpha - 2} \cos(\alpha \phi)
\quad\text{in } \Omega.
\]
Along the line $\phi = 0$, we have $\delta(r, 0) = r$ for sufficiently small $r$, and
\[
\frac{u(r, 0)}{ \delta(r, 0)^{ \alpha } }
=\frac{1}{2\alpha} \log \frac{1}{r}.
\]
Here $p=2$, $w=1$, and the datum satisfies \eqref{eqn:morrey}
with $\lambda=2-\alpha$, hence $\beta=\alpha$: for $r'<\delta(x)/2$,
its density on $B(x,r')$ is at most $C|x|^{\alpha-2}\le C(r')^{\alpha-2}$.
The function $u$ vanishes on the whole boundary of the truncated sector.
Thus \eqref{eqn:main-bound} fails at $\alpha_{ \mathrm{ss} }=\beta=\alpha$.
\end{remark}

For the full-boundary results below, take $\Gamma=\partial\Omega$ and write $\delta=\delta_\Gamma$.


For $0<\alpha\le1$ and $F\subset\R^n$, write
\[
[u]_{C^\alpha(F)}:=\sup_{x,z\in F,\ x\ne z}
\frac{|u(x)-u(z)|}{|x-z|^\alpha}.
\]

\begin{theorem}\label{thm:local}
Assume that \eqref{eqn:cdc} holds on $\partial\Omega$.
Let $0\le\lambda<p$ and let $\nu$
be a signed locally finite Radon measure on $\Omega$ satisfying
\eqref{eqn:morrey} with $\Gamma=\partial\Omega$ and
$\|\delta\|_{L^\infty(\spt\nu)}<\infty$.
Then there exists a bounded weak solution
$u\in H^{1,p}_{\loc}(\Omega;w)\cap C(\cl\Omega)$ to
\[
\begin{cases}
-\laplacian_{p,w}u=\nu & \text{in }\Omega,\\
u=0 & \text{on }\partial\Omega.
\end{cases}
\]
This solution satisfies \eqref{eqn:main-bound} with
$\Gamma=\partial\Omega$ and $\alpha_{\mathrm{ss}}$ as in
Proposition \ref{prop:barrier}. Moreover,
\begin{equation}\label{eqn:holder}
[u]_{C^{\alpha_{\mathrm{gh}}}(\cl\Omega)}
\le CM^{1/(p-1)}\|\delta\|_{L^\infty(\spt\nu)}^{\beta-\alpha_{\mathrm{gh}}},
\end{equation}
where
$\alpha_{\mathrm{gh}}=\min\{\alpha_{\mathrm{in}},\alpha_{\mathrm{ss}}\}$
and $\alpha_{\mathrm{in}}$ is the exponent in
Lemma \ref{lem:interior}.
The constants depend only on the structural constants, $\lambda$
and $\gamma$.
If $\nu\ge0$, the solution can be chosen to be minimal among
nonnegative solutions.
\end{theorem}

\begin{proof}
If $\nu=0$, take $u=0$. Suppose henceforth that $\nu\ne0$.
By Proposition \ref{prop:barrier}, there exists a nonnegative
$S\in H^{1,p}_{\loc}(\Omega;w)\cap C(\Omega)$ satisfying
\eqref{eqn:barrier2} with $\Gamma=\partial\Omega$ and
\[
-\laplacian_{p,w}S\ge|\nu|\quad\text{in }\Omega.
\]
In particular, $S$ is bounded and extends continuously by zero
to $\partial\Omega$.

For each integer $j\ge1$, consider the Dirichlet problem
\begin{equation}\label{eqn:local-approximation}
\begin{cases}
-\laplacian_{p,w}u_j
=\mathbf1_{\{|x|<j,\ \delta(x)>1/j\}}\nu
& \text{in }\Omega\cap B(0,2j),\\
u_j=0 & \text{on }\partial(\Omega\cap B(0,2j)).
\end{cases}
\end{equation}
The restricted measure has compact support in $\Omega\cap B(0,2j)$.
A finite covering of its support by interior balls, together with
\eqref{eqn:morrey} and \eqref{eqn:trace}, shows that it belongs to
$(H^{1,p}_0(\Omega\cap B(0,2j);w))^*$.
Thus \eqref{eqn:local-approximation} has a unique weak solution
$u_j\in H^{1,p}_0(\Omega\cap B(0,2j);w)$.
Lemma \ref{lem:interior} gives continuity in $\Omega\cap B(0,2j)$.
The datum vanishes near the boundary of this set.
At boundary points belonging to $\partial\Omega$, its complement
contains $\R^n\setminus\Omega$; at the remaining boundary points,
its complement contains the exterior of $B(0,2j)$.
Consequently, the capacity density condition holds locally at each
boundary point. Lemma \ref{lem:boundary}, applied to the homogeneous
equation, gives continuous zero boundary values for $u_j$.

For $\epsilon>0$, the functions $(\pm u_j-S-\epsilon)_+$ have
compact support in $\Omega\cap B(0,2j)$.
The comparison principle applied to
\eqref{eqn:local-approximation} and \eqref{eqn:barrier} for $|\nu|$
therefore gives, on letting $\epsilon\downarrow0$,
\begin{equation}\label{eqn:local-approximation-bound}
|u_j(x)|\le S(x),
\end{equation}
for all $x\in\Omega\cap B(0,2j)$.

Let $0\le\eta\le1$ belong to $C_c^\infty(\Omega)$.
For all sufficiently large $j$, the right-hand side of
\eqref{eqn:local-approximation} equals $\nu$ in a neighborhood
of $\spt\eta$.
Using $u_j\eta^p$ as a test function in
\eqref{eqn:local-approximation}, we obtain
\begin{equation}\label{eqn:local-approximation-energy}
\int_\Omega\eta^p|\nabla u_j|^p\,dw
\le C\int_\Omega |u_j|^p|\nabla\eta|^p\,dw
+C\int_\Omega\eta^p|u_j|\,d|\nu|.
\end{equation}
By \eqref{eqn:local-approximation-bound}, the right-hand side is
bounded independently of $j$.
Together with \eqref{eqn:interior}, this gives a subsequence and
$u\in H^{1,p}_{\loc}(\Omega;w)\cap C(\Omega)$ such that
\begin{equation}\label{eqn:local-approximation-limit}
\begin{aligned}
u_j&\longrightarrow u &&\text{locally uniformly in }\Omega,\\
\nabla u_j&\rightharpoonup\nabla u
&&\text{weakly in }L^p_{\loc}(\Omega;w).
\end{aligned}
\end{equation}
The standard convergence argument for monotone operators
\cite[Theorem 3.75]{MR2305115} applies here: the equations have
the same datum on every fixed compact subset for sufficiently
large $j$, and the measure term obtained by testing with
$\eta(u_j-u)$ tends to zero by
\eqref{eqn:local-approximation-limit}.
It follows that
\[
\int_\Omega |\nabla u|^{p-2}\nabla u\cdot\nabla\varphi\,dw
=\int_\Omega\varphi\,d\nu,
\]
for all $\varphi\in C_c^\infty(\Omega)$.
Thus $-\laplacian_{p,w}u=\nu$ in $\Omega$.
Passing to the limit in \eqref{eqn:local-approximation-bound}
and using \eqref{eqn:barrier2}, we obtain
\eqref{eqn:main-bound}. This estimate also gives
$u\in C(\cl\Omega)$ and $u=0$ on $\partial\Omega$.

We next prove \eqref{eqn:holder}.
Let $x,z\in\Omega$.
If $|x-z|\ge\|\delta\|_{L^\infty(\spt\nu)}/8$, then
\eqref{eqn:main-bound} gives
\[
|u(x)-u(z)|
\le CM^{1/(p-1)}\|\delta\|_{L^\infty(\spt\nu)}^\beta
\le CM^{1/(p-1)}\|\delta\|_{L^\infty(\spt\nu)}^{\beta-\alpha_{\mathrm{gh}}}
|x-z|^{\alpha_{\mathrm{gh}}}.
\]
Suppose that $|x-z|<\|\delta\|_{L^\infty(\spt\nu)}/8$.
If $\min\{\delta(x),\delta(z)\}\le4|x-z|$, then
$\max\{\delta(x),\delta(z)\}\le5|x-z|$.
Hence \eqref{eqn:main-bound} yields
\[
\begin{aligned}
|u(x)-u(z)|
&\le CM^{1/(p-1)}\|\delta\|_{L^\infty(\spt\nu)}^{\beta-\alpha_{\mathrm{ss}}}
|x-z|^{\alpha_{\mathrm{ss}}}\\
&\le CM^{1/(p-1)}\|\delta\|_{L^\infty(\spt\nu)}^{\beta-\alpha_{\mathrm{gh}}}
|x-z|^{\alpha_{\mathrm{gh}}}.
\end{aligned}
\]

It remains to consider the case
\[
|x-z|<\|\delta\|_{L^\infty(\spt\nu)}/8,
\qquad 4|x-z|<\min\{\delta(x),\delta(z)\}.
\]
Take
\[
R=\frac14\min\{\delta(x),\|\delta\|_{L^\infty(\spt\nu)}\}.
\]
Then $|x-z|<R$ and $B(x,2R)\subset\Omega$.
The restriction of $\nu$ to $B(x,2R)$, extended by zero,
satisfies \eqref{eqn:full} with $R_0=R$ and constant $CM$.
Indeed, balls of radius less than $R/4$ meeting this restriction
can be recentered in $B(x,2R)$ and estimated by
\eqref{eqn:morrey}; larger balls are estimated by covering
$B(x,2R)$ with balls of radius $R/8$ and using doubling.
We may therefore apply \eqref{eqn:interior} to
$-\laplacian_{p,w}u=\nu$ in $B(x,2R)$.
Since $\alpha_{\mathrm{gh}}\le\alpha_{\mathrm{in}}$, this gives
\begin{equation}\label{eqn:local-holder-interior-step}
|u(x)-u(z)|
\le C\left(\frac{|x-z|}{R}\right)^{\alpha_{\mathrm{gh}}}
\left(\osc_{B(x,2R)}u+M^{1/(p-1)}R^\beta\right).
\end{equation}
If $\delta(x)\le\|\delta\|_{L^\infty(\spt\nu)}$, then
$\delta(y)\le6R$ for $y\in B(x,2R)$.
If $\delta(x)>\|\delta\|_{L^\infty(\spt\nu)}$, then
$\|\delta\|_{L^\infty(\spt\nu)}=4R$.
In both cases, \eqref{eqn:main-bound} gives
\[
\osc_{B(x,2R)}u
\le CM^{1/(p-1)}\|\delta\|_{L^\infty(\spt\nu)}^{\beta-\alpha_{\mathrm{ss}}}
R^{\alpha_{\mathrm{ss}}}.
\]
Substituting this estimate into
\eqref{eqn:local-holder-interior-step}, and using
$R\le\|\delta\|_{L^\infty(\spt\nu)}$ and
$\alpha_{\mathrm{gh}}\le\alpha_{\mathrm{ss}}<\beta$, we obtain
\[
|u(x)-u(z)|
\le CM^{1/(p-1)}\|\delta\|_{L^\infty(\spt\nu)}^{\beta-\alpha_{\mathrm{gh}}}
|x-z|^{\alpha_{\mathrm{gh}}}.
\]
Continuity on $\cl\Omega$ now proves \eqref{eqn:holder}.

Finally, suppose that $\nu\ge0$.
The comparison principle for \eqref{eqn:local-approximation} gives
$0\le u_j\le u_{j+1}$ in $\Omega\cap B(0,2j)$.
If $v$ is any nonnegative solution in $\Omega$, comparison on
$\Omega\cap B(0,2j)$ also gives $u_j\le v$ there.
Letting $j\to\infty$, we obtain $u\le v$ in $\Omega$.
Thus $u$ is the minimal nonnegative solution.
\end{proof}

\begin{theorem}\label{thm:local-uniqueness}
Let $\nu$ be a signed locally finite measure on $\Omega$.
Assume that $\Omega$ is bounded, or that $p=2$ and \eqref{eqn:cdc}
holds on $\partial\Omega$. Then \eqref{eqn:DE} has at most one bounded
weak solution in $H^{1,p}_{\loc}(\Omega;w)\cap C(\cl\Omega)$.
\end{theorem}
\begin{proof}
Let $u_1$ and $u_2$ be two such solutions. If $\Omega$ is bounded,
$(u_1-u_2-\epsilon)_+$ has compact support in $\Omega$ for each
$\epsilon>0$, because both solutions vanish continuously on $\partial\Omega$.
Subtracting their weak equations and using this function as a test gives
\[
\int_{\{u_1-u_2>\epsilon\}}
\bigl(|\nabla u_1|^{p-2}\nabla u_1-|\nabla u_2|^{p-2}\nabla u_2\bigr)
\cdot(\nabla u_1-\nabla u_2)\,dw=0.
\]
Strict monotonicity implies $(u_1-u_2-\epsilon)_+=0$.
Letting $\epsilon\downarrow0$ and interchanging $u_1$ and $u_2$ proves uniqueness.

Suppose now that $p=2$. The difference $v=u_1-u_2$ satisfies
\begin{equation}\label{eqn:uniqueness-harmonic}
-\divergence(w\nabla v)=0\quad\text{in }\Omega,
\qquad v=0\quad\text{on }\partial\Omega.
\end{equation}
For $\eta\in C_c^\infty(\R^n)$, the tests
$\eta^2(v-\epsilon)_+$ and $-\eta^2(-v-\epsilon)_+$ have compact
support in $\Omega$. The Caccioppoli estimate for
\eqref{eqn:uniqueness-harmonic}, followed by $\epsilon\downarrow0$, yields
\[
\int_\Omega\eta^2|\nabla v|^2\,dw
\le C\int_\Omega |v|^2|\nabla\eta|^2\,dw.
\]
The same truncations show that $v$ has local zero Sobolev boundary values.
Fix $\xi\in\partial\Omega$ and $r>0$. Applying \eqref{eqn:boundary}
to the positive and negative parts of $v$, with zero datum, gives,
for every $R>8r$,
\[
\sup_{\Omega\cap B(\xi,r)}|v|
\le C(r/R)^{\alpha_{\mathrm{bd}}}\|v\|_{L^\infty(\Omega)}.
\]
The constant is independent of $R$. Hence $v=0$ in
$\Omega\cap B(\xi,r)$ by letting $R\to\infty$.
Since $r>0$ is arbitrary, $u_1=u_2$ in $\Omega$.
\end{proof}

\begin{theorem}\label{thm:capacity-lower-bound}
Fix $0\le\lambda<p$.
Suppose that every nonnegative measure satisfying \eqref{eqn:morrey}
with $\Gamma=\partial\Omega$ and $\|\delta\|_{L^\infty(\spt\nu)}<\infty$ has a nonnegative
weak solution to $-\laplacian_{p,w}u=\nu$ in $\Omega$ satisfying
\eqref{eqn:main-bound}, with fixed constants $C>0$ and
$0<\alpha_{\mathrm{ss}}<\min\{\beta,1\}$.
Then \eqref{eqn:cdc} holds on $\partial\Omega$.
Its constant $\gamma>0$ depends only on the structural constants,
$\lambda$, $C$ and $\alpha_{\mathrm{ss}}$.
\end{theorem}
\begin{proof}
Fix $\xi\in\partial\Omega$ and $R>0$, and consider the datum
\[
R^{-p}\mathbf1_{\Omega\cap(B(\xi,R/2)\setminus\cl{B(\xi,R/4)})}\,dw.
\]
If $0<r<\delta(x)/2$ and $B(x,r)$ meets its support, then
$2r<\delta(x)\le|x-\xi|<r+R/2$, so $r<R/2$.
Its mass on $B(x,r)$ is therefore at most
$R^{\lambda-p}r^{-\lambda}w(B(x,r))$.
The hypothesis gives a nonnegative solution $u_R$ with
\begin{equation}\label{eqn:annular-bound}
 u_R(x)\le C(\delta(x)/R)^{\alpha_{\mathrm{ss}}},
\end{equation}
for all $x\in\Omega$.
By translation and scaling, it suffices to prove a positive lower bound
for $\capacity_{p,w}(\cl{B(0,1)}\setminus\Omega,B(0,2))$ under the
normalization $\xi=0$, $R=1$ and $w(B(0,1))=1$.

Let $h$ be the capacity potential of $\cl{B(0,1)}\setminus\Omega$
in $B(0,2)$. Then $0\le h\le1$, $h=1$ quasi-everywhere on the
removed set, and
\begin{equation}\label{eqn:local-capacity-energy}
\int_{B(0,2)}|\nabla h|^p\,dw
=\capacity_{p,w}(\cl{B(0,1)}\setminus\Omega,B(0,2)).
\end{equation}
Let $z\in H^{1,p}_0(B(0,1);w)$ and
$v\in H^{1,p}_0(\Omega\cap B(0,1);w)$ solve, respectively,
\begin{equation}\label{eqn:local-reference-problems}
\begin{aligned}
-\laplacian_{p,w}z&=\mathbf1_{B(0,1/2)\setminus\cl{B(0,1/4)}}w
&&\text{in }B(0,1),\\
-\laplacian_{p,w}v&=\mathbf1_{B(0,1/2)\setminus\cl{B(0,1/4)}}w
&&\text{in }\Omega\cap B(0,1).
\end{aligned}
\end{equation}
We extend $v$ by zero to $B(0,1)$.
Comparison on relatively compact subdomains of $\Omega\cap B(0,1)$,
followed by energy convergence of their zero-boundary solutions,
gives $0\le v\le u_1$ in $\Omega\cap B(0,1)$.

The reference solution satisfies
\begin{equation}\label{eqn:reference-local}
0\le z\le C,
\qquad z\ge c>0,
\end{equation}
in $B(0,1)$ and $B(0,1/2)$, respectively.
The upper bound follows from \eqref{eqn:maximum}.
For the lower bound, test \eqref{eqn:local-reference-problems} with a
fixed nonnegative cutoff supported in the annulus and equal to one
on a smaller ball. Doubling and H\"older's inequality give
$\int_{B(0,1)}|\nabla z|^p\,dw\ge c$.
Testing with $z$ then gives
$\int_{B(0,1/2)\setminus\cl{B(0,1/4)}}z\,dw\ge c$.
The upper bound in \eqref{eqn:reference-local} and the homogeneous weak
Harnack inequality for the nonnegative supersolution $z$ imply the
stated lower bound.

The function $z(1-h)$ belongs to $H^{1,p}_0(\Omega\cap B(0,1);w)$.
Indeed, it has zero Sobolev values on $\partial B(0,1)$ and vanishes
quasi-everywhere on $B(0,1)\setminus\Omega$; see
\cite[Chapters 4 and 21]{MR2305115}.
Testing the equation for $z$ with $zh^p$, and using
\eqref{eqn:reference-local} and the Poincar\'e inequality for $h$, gives
\begin{equation}\label{eqn:local-capacity-product}
\int_{B(0,1)}|\nabla(zh)|^p\,dw
\le C\int_{B(0,2)}|\nabla h|^p\,dw.
\end{equation}
The energy functional for \eqref{eqn:local-reference-problems} is minimized
by $z$ on $H^{1,p}_0(B(0,1);w)$ and by $v$ on its zero-boundary subspace
for $\Omega\cap B(0,1)$. We use $z(1-h)$ as a competitor for $v$.
By \eqref{eqn:local-capacity-product}, the difference of the two minimum
values is at most
$C(\int_{B(0,2)}|\nabla h|^p\,dw)^{1/p}$ when that integral is at most one.
Uniform convexity of the $p$-energy and the energy bounds for $z$ and $v$
therefore give
\begin{equation}\label{eqn:energy-gap}
\int_{B(0,1)}|v-z|^p\,dw
\le C\left(\int_{B(0,2)}|\nabla h|^p\,dw\right)^{1/\max\{p,2\}}.
\end{equation}
Here we used the Poincar\'e inequality after the convexity estimate.
For $1<p<2$, H\"older's inequality converts the weighted quadratic
convexity remainder into the square of the $L^p$ gradient norm.

Choose $0<\theta<1/2$ so that the upper bound in
\eqref{eqn:annular-bound} is less than half the lower bound in
\eqref{eqn:reference-local} on $\Omega\cap B(0,\theta)$.
Since $v=0$ off $\Omega$, these estimates and \eqref{eqn:energy-gap} yield
\[
c\,w(B(0,\theta))
\le C\left(\int_{B(0,2)}|\nabla h|^p\,dw\right)^{1/\max\{p,2\}}.
\]
Doubling and $w(B(0,1))=1$ bound the left-hand side below.
Equation \eqref{eqn:local-capacity-energy} gives the required positive
capacity bound. If the capacity energy exceeds one, the bound is immediate.
Scaling back proves \eqref{eqn:cdc} on $\partial\Omega$.
\end{proof}

\subsection{A necessary condition on the measure}
The next lemma is the weighted version of \cite[Theorem 3.2]{MR998128}.
\begin{lemma}\label{lem:necessary-growth}
Let $\nu\ge0$ be locally finite on $\Omega$ and let
$u\in H^{1,p}_{\loc}(\Omega;w)\cap C^\alpha(\cl\Omega)$ satisfy
$-\laplacian_{p,w}u=\nu$ in $\Omega$, where $0<\alpha\le1$.
Then
\begin{equation}\label{eqn:necessary-growth}
\nu(B(x,r))\le C[u]_{C^\alpha(\cl\Omega)}^{p-1}
r^{(p-1)\alpha-p}w(B(x,r))
\end{equation}
for every $x\in\Omega$ and $0<r<\delta(x)/2$.
The constant depends only on $p$ and the doubling constant of $w$.
\end{lemma}
\begin{proof}
Fix $x\in\Omega$ and first suppose that $0<r<\delta(x)/4$.
Choose $\eta\in C_c^\infty(B(x,3r))$ equal to one on $B(x,2r)$,
with $0\le\eta\le1$ and $|\nabla\eta|\le C/r$.
The test $(\sup_{B(x,3r)}u-u)\eta^p$ is nonnegative.
The weak equation and $\nu\ge0$ imply
\begin{equation}\label{eqn:necessary-caccioppoli}
\int_{B(x,2r)}|\nabla u|^p\,dw
\le Cr^{-p}(\osc_{B(x,3r)}u)^p w(B(x,3r)).
\end{equation}
The test is admissible by local Sobolev approximation, since the equation
makes $\nu$ a continuous functional on $H^{1,p}_0(B(x,3r);w)$.
Testing once more with a cutoff equal to one on $B(x,r)$ and supported
in $B(x,2r)$, we obtain from \eqref{eqn:necessary-caccioppoli}
\[
\begin{aligned}
\nu(B(x,r))
&\le Cr^{-1}w(B(x,2r))^{1/p}
\left(\int_{B(x,2r)}|\nabla u|^p\,dw\right)^{(p-1)/p}
\\
&\le Cr^{-p}(\osc_{B(x,3r)}u)^{p-1}w(B(x,r)).
\end{aligned}
\]
The H\"older bound for $u$ proves \eqref{eqn:necessary-growth} in this range.
For $r<\delta(x)/2$, cover $B(x,r)$ by finitely many balls of radius
$r/8$ with centers in $B(x,r)$. Each smaller ball lies in the range
already proved. Summation and doubling give \eqref{eqn:necessary-growth}.
\end{proof}

\begin{corollary}\label{cor:data-criterion}
Assume \eqref{eqn:cdc} on $\partial\Omega$. Let $\nu\ge0$ be locally
finite on $\Omega$ with $\|\delta\|_{L^\infty(\spt\nu)}<\infty$.
Problem \eqref{eqn:DE} has a bounded solution in
$H^{1,p}_{\loc}(\Omega;w)\cap C^\alpha(\cl\Omega)$ for some $0<\alpha\le1$
if and only if \eqref{eqn:morrey}, with $\Gamma=\partial\Omega$,
holds for some $0\le\lambda<p$ and $M<\infty$.
Among bounded solutions, uniqueness holds if $\Omega$ is bounded or $p=2$.
\end{corollary}
\begin{proof}
If such a solution exists, \eqref{eqn:necessary-growth} gives
\eqref{eqn:morrey} with $\lambda=p-(p-1)\alpha$.
Conversely, Theorem \ref{thm:local} gives a solution satisfying
\eqref{eqn:holder}. Uniqueness follows from Theorem \ref{thm:local-uniqueness}.
\end{proof}

\subsection{Trace estimates and compactness}
\begin{theorem}\label{thm:compact}
Let $\Omega$ be open, $w$ be $p$-admissible, and $\nu\ge0$ be locally
finite on $\Omega$ with bounded support. Suppose that
$S\in H^{1,p}_{\loc}(\Omega;w)\cap C(\cl\Omega)$ satisfies
\[
S\ge0\quad\text{in }\Omega,\qquad S=0\quad\text{on }\partial\Omega,
\qquad -\laplacian_{p,w}S\ge\nu\quad\text{in }\Omega.
\]
Then the identity on $C_c^\infty(\Omega)$ extends to a compact trace
$T_\nu:H^{1,p}_0(\Omega;w)\to L^p(\Omega;\nu)$.
For every $\epsilon>0$, there is $C_\epsilon>0$ such that
\begin{equation}\label{eqn:compact-local}
\|T_\nu v\|_{L^p(\nu)}^p
\le\epsilon\int_\Omega|\nabla v|^p\,dw
+C_\epsilon\int_\Omega|v|^p\,dw,
\end{equation}
for all $v\in H^{1,p}_0(\Omega;w)$.
\end{theorem}
\begin{proof}
Let $B(x,2r)\Subset\Omega$. Picone's inequality \cite{MR1618334},
applied to $S-\inf_{B(x,2r)}S+t$, gives
\begin{equation}\label{eqn:picone}
\int_{B(x,2r)}|\varphi|^p\,d\nu
\le(\osc_{B(x,2r)}S)^{p-1}
\int_{B(x,2r)}|\nabla\varphi|^p\,dw
\end{equation}
for $\varphi\in C_c^\infty(B(x,2r))$, after letting $t\downarrow0$.
The test is $|\varphi|^p/(S-\inf_{B(x,2r)}S+t)^{p-1}$.
It is justified by local Sobolev approximation and
$0\le\nu\le-\laplacian_{p,w}S$.
On a compact subset of $\Omega$, uniform continuity of $S$ makes
the coefficient in \eqref{eqn:picone} arbitrarily small when $r$ is small.
Applying \eqref{eqn:picone} to cutoffs of $v$ and summing over a covering
with bounded overlap therefore gives the desired estimate on that subset.

The same test with $S+t$ on $\Omega$ gives
\[
\int_\Omega\frac{|v|^p}{(S+t)^{p-1}}\,d\nu
\le\int_\Omega|\nabla v|^p\,dw,
\]
for all $v\in C_c^\infty(\Omega)$. Restricting the left-hand side and letting $t\downarrow0$, we obtain
\begin{equation}\label{eqn:compact-boundary-layer}
\int_{\Omega\cap\{\delta<h\}}|v|^p\,d\nu
\le\left(\sup_{\spt\nu\cap\{\delta<h\}}S\right)^{p-1}
\int_\Omega|\nabla v|^p\,dw.
\end{equation}
The coefficient tends to zero as $h\downarrow0$, since $S$ is continuous,
vanishes on $\partial\Omega$, and $\spt\nu$ is bounded.
The supremum is taken to be zero if its set is empty.
Combining \eqref{eqn:compact-boundary-layer} with the interior estimate
proves \eqref{eqn:compact-local} for smooth $v$. It also gives the continuous
extension $T_\nu$, and density proves the estimate on $H^{1,p}_0(\Omega;w)$.

Choose $\eta\in C_c^\infty(\R^n)$ equal to one near $\spt\nu$.
For a bounded sequence $(v_j)$ in $H^{1,p}_0(\Omega;w)$, the zero extensions
of $\eta v_j$ have bounded Sobolev norms and support in a fixed ball.
The weighted Poincar\'e inequality gives approximation in $L^p(w)$ by
averages on finitely many balls, with error bounded by
$Cr\|\nabla(\eta v_j)\|_{L^p(w)}$.
Thus a subsequence of $(\eta v_j)$ converges in $L^p(w)$.
Applying \eqref{eqn:compact-local} to its differences, and then letting
$\epsilon\downarrow0$, proves convergence in $L^p(\nu)$.
Since $T_\nu(\eta v_j)=T_\nu v_j$, the trace is compact.
\end{proof}

\begin{corollary}\label{cor:compact-local-distance}
Assume \eqref{eqn:cdc} on $\partial\Omega$.
Let $0\le\lambda<p$, $M>0$, and let $\nu\ge0$ be locally finite on
$\Omega$, satisfying
\begin{equation}\label{eqn:local-trace-growth}
\nu(B(x,r))\le Mr^{-\lambda}w(B(x,r))
\quad(x\in\Omega,\ 0<r<\delta(x)/2).
\end{equation}
For $0\le\vartheta<p-\lambda$, the identity on $C_c^\infty(\Omega)$
extends to a continuous trace
$T_\nu:H^{1,p}_0(\Omega;w)\to L^p(\Omega;\delta^{-\vartheta}d\nu)$.
If $\|\delta\|_{L^\infty(\spt\nu)}<\infty$, then
\begin{equation}\label{eqn:depth-trace}
\int_\Omega|T_\nu u|^p\delta^{-\vartheta}\,d\nu
\le CM\|\delta\|_{L^\infty(\spt\nu)}^{p-\lambda-\vartheta}
\int_\Omega|\nabla u|^p\,dw.
\end{equation}
For every $\epsilon>0$,
\begin{equation}\label{eqn:epsilon}
\begin{aligned}
\int_\Omega|T_\nu u|^p\delta^{-\vartheta}\,d\nu
\le{}&\epsilon\int_\Omega|\nabla u|^p\,dw\\
&+CM^{p/(p-\lambda-\vartheta)}
\epsilon^{-(\lambda+\vartheta)/(p-\lambda-\vartheta)}
\int_\Omega|u|^p\,dw.
\end{aligned}
\end{equation}
Both estimates hold for $u\in H^{1,p}_0(\Omega;w)$.
The trace is compact if $\spt\nu$ is bounded.
The constants depend only on the structural constants, $\gamma$,
$\lambda$ and $\vartheta$.
\end{corollary}
\begin{proof}
Fix $u\in C_c^\infty(\Omega)$.
The assertion is immediate for $\nu=0$. Otherwise,
$\delta^{-\vartheta}\nu$ satisfies \eqref{eqn:morrey} with exponent
$\lambda+\vartheta<p$ and coefficient $CM$.
If its support has finite depth, Proposition \ref{prop:barrier} gives
$-\laplacian_{p,w}S\ge\delta^{-\vartheta}\nu$ with
\[
\|S\|_\infty^{p-1}
\le CM\|\delta\|_{L^\infty(\spt\nu)}^{p-\lambda-\vartheta}.
\]
Picone's inequality with $S+t$, followed by $t\downarrow0$, proves
\eqref{eqn:depth-trace} for smooth $u$.
Applying the same argument to
$\mathbf1_{\{\delta<h\}}\delta^{-\vartheta}\nu$, whose support has
depth at most $h$, gives
\begin{equation}\label{eqn:subcritical-boundary-layer}
\int_{\{\delta<h\}}|u|^p\delta^{-\vartheta}\,d\nu
\le CMh^{p-\lambda-\vartheta}\int_\Omega|\nabla u|^p\,dw.
\end{equation}
To estimate the interior contribution, take $8r<\delta(x)$.
By \eqref{eqn:local-trace-growth}, the restriction of $\nu$ to $B(x,2r)$
satisfies \eqref{eqn:full} at scale $r$, with the same exponent $\lambda$
and growth coefficient bounded by $CM$.
Applying \eqref{eqn:trace} to a cutoff of $u$ gives
\begin{equation}\label{eqn:local-trace-ball}
\int_{B(x,r)}|u|^p\,d\nu
\le CM\left(r^{p-\lambda}\int_{B(x,2r)}|\nabla u|^p\,dw
+r^{-\lambda}\int_{B(x,2r)}|u|^p\,dw\right).
\end{equation}
Cover $\{\delta\ge h\}$ by balls of radius $h/32$ whose doubled
balls have bounded overlap. On these balls $\delta^{-\vartheta}\le
Ch^{-\vartheta}$. Summing \eqref{eqn:local-trace-ball} and combining
with \eqref{eqn:subcritical-boundary-layer} yields
\[
\int_\Omega|u|^p\delta^{-\vartheta}\,d\nu
\le CM\left(h^{p-\lambda-\vartheta}\int_\Omega|\nabla u|^p\,dw
+h^{-\lambda-\vartheta}\int_\Omega|u|^p\,dw\right).
\]
Choosing $h$ so that $CMh^{p-\lambda-\vartheta}=\epsilon$ proves
\eqref{eqn:epsilon}. Density defines $T_\nu$ and extends both estimates
to $H^{1,p}_0(\Omega;w)$.
If $\spt\nu$ is bounded, the supersolution for
$\delta^{-\vartheta}\nu$ is continuous and vanishes on $\partial\Omega$;
Theorem \ref{thm:compact} gives compactness.
\end{proof}

\subsection{Two-weight Hardy inequalities}

Hardy inequalities under uniform fatness were studied by Ancona
\cite{MR856511}, Lewis \cite{MR946438} and Wannebo \cite{MR1010807};
Mikkonen \cite{MR1386213} treated the weighted setting.
For the capacity-density approach and distance weights, see
\cite[Chapters 6, 7 and 10]{MR4306765}.
The strong barriers in \cite{Hara2024} also give the inequality with
weight $\delta_\Gamma^{-p}\,dw$ for a closed part $\Gamma$ of the boundary.
The theorem below allows an independent measure satisfying a growth condition.

Let $\varnothing\ne\Gamma\subset\partial\Omega$ be closed.
\begin{theorem}\label{thm:two-weight}
Assume \eqref{eqn:cdc} on $\Gamma$.
Let $0\le\lambda<p$ and let $\nu$ be a nonnegative Radon measure on
$\R^n$ satisfying
\begin{equation}\label{eqn:global-growth}
\nu(B(x,r))\le M_{\mathrm g}r^{-\lambda}w(B(x,r))
\quad(x\in\R^n,\ r>0),
\end{equation}
where $M_{\mathrm g}>0$.
There exist $\alpha_{\mathrm{ss}}>0$ and
$U\in H^{1,p}_{\loc}(\Omega;w)\cap C(\Omega)$ such that
\begin{equation}\label{eqn:two-weight-barrier}
\begin{aligned}
\delta_\Gamma^{\alpha_{\mathrm{ss}}}\le U
&\le C\delta_\Gamma^{\alpha_{\mathrm{ss}}},\\
-\laplacian_{p,w}U
&\ge\frac{c}{M_{\mathrm g}}U^{p-1}\delta_\Gamma^{\lambda-p}\nu,
\end{aligned}
\end{equation}
in $\Omega$.
Consequently,
\begin{equation}\label{eqn:two-weight-hardy}
\int_\Omega|\varphi|^p\delta_\Gamma^{\lambda-p}\,d\nu
\le CM_{\mathrm g}\int_\Omega|\nabla\varphi|^p\,dw,
\end{equation}
for all $\varphi\in C_c^\infty(\Omega)$.
The constants depend only on the structural constants, $\lambda$ and $\gamma$.
\end{theorem}
\begin{proof}
We use the construction in Proposition \ref{prop:barrier} at all scales.
Choose $0<\alpha_{\mathrm{ss}}<\min\{(p-\lambda)/(p-1),\alpha_{\mathrm{bd}},1\}$
and then $\theta$ as in that proof.
For $B=B(\xi,R)$, $\xi\in\Gamma$, let $u_B$ solve
\begin{equation}\label{eqn:hardy-ball-equation}
-\laplacian_{p,w}u_B
=\frac{\epsilon}{M_{\mathrm g}}
R^{(p-1)\alpha_{\mathrm{ss}}+\lambda-p}\nu
\quad\text{in }\Omega\cap B,
\qquad u_B-\eta_B\in H^{1,p}_0(\Omega\cap B;w),
\end{equation}
where $\eta_B$ has the same boundary values as in Proposition \ref{prop:barrier}.
By \eqref{eqn:global-growth}, the datum satisfies \eqref{eqn:full}
with the same exponent $\lambda$. Its contribution to \eqref{eqn:maximum}
is bounded by $C\epsilon^{1/(p-1)}R^{\alpha_{\mathrm{ss}}}$.
Thus \eqref{eqn:maximum} and \eqref{eqn:boundary} give
\begin{equation}\label{eqn:hardy-ball-bounds}
\frac14(\theta R)^{\alpha_{\mathrm{ss}}}\le u_B\le2R^{\alpha_{\mathrm{ss}}},
\qquad
u_B\le\frac12(\theta R)^{\alpha_{\mathrm{ss}}},
\end{equation}
in $\Omega\cap B$ and $\Omega\cap2\theta B$, respectively,
provided $\epsilon>0$ is sufficiently small, independently of $R$.
We truncate $u_B$ at $R^{\alpha_{\mathrm{ss}}}$ and extend it by that
constant outside $B$, obtaining $\widehat u_B$ as in Proposition
\ref{prop:barrier}.

For each $k\in\mathbb Z$, we choose the boundary covering used in
Proposition \ref{prop:barrier}, now with radius $\theta^k$.
We define $v_k$ as the locally finite minimum of the corresponding
functions $\widehat u_B$, as in that construction,
and set
\[
U(x)=\inf\{v_k(x):k\in\mathbb Z,\ \delta_\Gamma(x)<\theta^k\}.
\]
On compact subsets of $\Omega$, sufficiently large positive indices are
excluded by the distance condition; sufficiently large negative indices
are excluded by the lower bound in \eqref{eqn:hardy-ball-bounds}.
The infimum is therefore locally a minimum of finitely many functions.
At $\delta_\Gamma=\theta^k$, the inequality
$v_{k-1}\le\tfrac12\theta^{k\alpha_{\mathrm{ss}}}<v_k$
excludes the changing index in a neighborhood, as in Proposition
\ref{prop:barrier}.
Consequently $U\in H^{1,p}_{\loc}(\Omega;w)\cap C(\Omega)$.

If $\theta^{k+1}<\delta_\Gamma(x)\le\theta^k$, the indices attaining
the minimum are among $k-1$ and $k$. Each $v_i$ attaining the minimum
lies strictly below its truncation level, so only untruncated ball
solutions contribute locally.
Equations \eqref{eqn:hardy-ball-equation} and
\eqref{eqn:hardy-ball-bounds}, together with the local minimum argument
of Proposition \ref{prop:barrier}, give
\[
C^{-1}\delta_\Gamma^{\alpha_{\mathrm{ss}}}\le U
\le C\delta_\Gamma^{\alpha_{\mathrm{ss}}},
\qquad
-\laplacian_{p,w}U\ge
\frac{c}{M_{\mathrm g}}\delta_\Gamma^{(p-1)\alpha_{\mathrm{ss}}+\lambda-p}\nu,
\]
in $\Omega$.
The neighborhood argument at each scale boundary also covers any mass of
$\nu$ carried by a level set of $\delta_\Gamma$.
Multiplying $U$ by a fixed constant proves \eqref{eqn:two-weight-barrier}.

For $\varphi\in C_c^\infty(\Omega)$, test
\eqref{eqn:two-weight-barrier} with $|\varphi|^p(U+t)^{1-p}$.
Local trace continuity justifies the test, and Picone's inequality gives
\[
\frac{c}{M_{\mathrm g}}
\int_\Omega|\varphi|^p\left(\frac{U}{U+t}\right)^{p-1}
\delta_\Gamma^{\lambda-p}\,d\nu
\le\int_\Omega|\nabla\varphi|^p\,dw.
\]
Since $U>0$ in $\Omega$, letting $t\downarrow0$ proves
\eqref{eqn:two-weight-hardy}.
\end{proof}

\begin{corollary}\label{cor:local-hardy}
Assume \eqref{eqn:cdc} on $\Gamma$. Then
\begin{equation}\label{eqn:local-hardy}
\int_\Omega|\varphi|^p\delta_\Gamma^{-p}\,dw
\le C\int_\Omega|\nabla\varphi|^p\,dw,
\end{equation}
for all $\varphi\in C_c^\infty(\Omega)$.
\end{corollary}
\begin{proof}
Apply \eqref{eqn:two-weight-hardy} with $d\nu=dw$, $\lambda=0$ and
$M_{\mathrm g}=1$.
\end{proof}
Condition \eqref{eqn:global-growth} is imposed on all balls, whereas
\eqref{eqn:morrey} is imposed for $x\in\Omega$ and
$0<r<\delta_\Gamma(x)/2$, with the measure evaluated on $B(x,r)\cap\Omega$.

At the endpoint $\vartheta=p-\lambda$, the Hardy inequality gives
an additional trace estimate.
\begin{corollary}\label{thm:quantitative}
Assume \eqref{eqn:cdc} on $\partial\Omega$.
Let $0\le\lambda<p$, $M>0$, and let $\nu\ge0$ be locally finite on
$\Omega$, satisfying \eqref{eqn:local-trace-growth}.
The identity on $C_c^\infty(\Omega)$ extends to a continuous trace
$T_\nu:H^{1,p}_0(\Omega;w)\to L^p(\Omega;\delta^{\lambda-p}d\nu)$,
and
\begin{equation}\label{eqn:critical}
\int_\Omega|T_\nu u|^p\delta^{\lambda-p}\,d\nu
\le CM\int_\Omega|\nabla u|^p\,dw,
\end{equation}
for all $u\in H^{1,p}_0(\Omega;w)$.
The constant depends only on the structural constants, $\gamma$ and $\lambda$.
\end{corollary}
\begin{proof}
Let $u\in C_c^\infty(\Omega)$. Take a Whitney covering whose radii are
a sufficiently small fixed multiple of the boundary distance and whose
doubled balls have bounded overlap.
Multiplying \eqref{eqn:local-trace-ball} by $r^{\lambda-p}$ and summing gives
\[
\int_\Omega|u|^p\delta^{\lambda-p}\,d\nu
\le CM\left(\int_\Omega|\nabla u|^p\,dw
+\int_\Omega|u|^p\delta^{-p}\,dw\right).
\]
Equation \eqref{eqn:local-hardy} bounds the last integral.
Density defines the trace and proves \eqref{eqn:critical}.
It agrees locally $\nu$-almost everywhere with the traces in
Corollary \ref{cor:compact-local-distance}.
\end{proof}

\section{Fractional Hardy inequalities}\label{sec:fractional-hardy}
Let $n\ge1$, $1<p<\infty$, $0<s<1$, and let $D\subsetneq\R^n$ be
nonempty and open. For $A\subset\R^n$, we write
$\complement A=\R^n\setminus A$. We write $\delta(x)=\dist(x,\partial D)$
and extend functions on $D$ by zero to $\complement D$.
The fractional seminorm and the relative capacity are defined by
\[
[u]_{s,p}^p=\iint_{\R^n\times\R^n}
\frac{|u(x)-u(z)|^p}{|x-z|^{n+ps}}\,dx\,dz,
\]
\[
\capacity_{s,p}(K,B)=\inf\{[u]_{s,p}^p:
 u\in C_c^\infty(B),\ u\ge1\text{ on }K\}.
\]
Here $K\subset B$ is compact.
We write $W^{s,p}(\R^n)=\{u\in L^p(\R^n):[u]_{s,p}<\infty\}$,
with norm $(\|u\|_{L^p(\R^n)}^p+[u]_{s,p}^p)^{1/p}$.
Throughout this section, the extension domain and its weight are
\begin{equation}\label{eqn:slit-p}
O=\R^{n+1}\setminus(\complement D\times\{0\}),
\qquad w(x,y)=|y|^{p-1-ps}.
\end{equation}
The weight $w$ is $p$-admissible. We use the distance notation
$\delta_\Gamma$ also for nonempty closed sets $\Gamma\subset\complement D$.

\subsection{Trace estimates and the Hardy inequality}\label{sec:fractional-trace}
Choose an even nonnegative $\rho\in C_c^\infty(B(0,1))$ with integral one.
For $u\in C_c^\infty(D)$ and $(x,y)\in\R^{n+1}$, define
\[
Eu(x,y)=\int_{\R^n}\rho(z)u(x-yz)\,dz.
\]
The derivative kernels have mean zero. Hence Jensen's inequality gives
\begin{equation}\label{eqn:extension-derivative}
|\nabla Eu(x,y)|^p\le C|y|^{-p-n}
\int_{B(0,|y|)}|u(x-h)-u(x)|^p\,dh,
\end{equation}
for $y\ne0$.
Integrating \eqref{eqn:extension-derivative} with respect to $w\,dX$
and applying Tonelli's theorem, we obtain
\begin{equation}\label{eqn:extension-p}
\begin{aligned}
\int_{\R^{n+1}}w|\nabla Eu|^p\,dX&\le C[u]_{s,p}^p,\\
\int_{\R^n\times(-R,R)}w|Eu|^p\,dX&\le CR^{p-ps}\|u\|_{L^p(\R^n)}^p,
\end{aligned}
\end{equation}
for every $R>0$.
The second estimate follows directly from Jensen's inequality in the
formula for $Eu$. Since $u$ has compact support in $D$, $Eu$ vanishes
in a neighborhood of $\complement D\times\{0\}$.
Let $\eta_R$ be a smooth radial cutoff equal to one on $B(0,R)$,
supported in $B(0,2R)$, and satisfying $|\nabla\eta_R|\le C/R$.
By \eqref{eqn:extension-p},
\begin{equation}\label{eqn:extension-cutoff}
\int_{\R^{n+1}}w|Eu\nabla\eta_R|^p\,dX
\le CR^{-ps}\|u\|_{L^p(\R^n)}^p\longrightarrow0.
\end{equation}
Thus $\eta_REu\in C_c^\infty(O)$ converges to $Eu$ in the gradient norm,
and $(\eta_REu)(x,0)=u(x)$ for all sufficiently large $R$.

We also use the homogeneous trace inequality
\begin{equation}\label{eqn:homogeneous-trace}
[V(\cdot,0)]_{s,p}^p\le C\int_{\R^{n+1}}w|\nabla V|^p\,dX,
\end{equation}
for all $V\in C_c^\infty(\R^{n+1})$.
It follows from \cite[Theorem 1.1]{KoskelaSotoWang} by dilation and
letting the scale tend to infinity. The scaled lower-order term is $R^{-p}\int_{\R^{n+1}}w|V|^p\,dX$. The same estimate holds in the upper half-space.

For the identification of the extensions in Sections
\ref{sec:fractional} and \ref{sec:nonlinear}, we also need approximation
of zero-trace functions. The one-dimensional Hardy inequality gives
\begin{equation}\label{eqn:zero-trace-hardy-p}
\int_{\R^n\times(0,\infty)}y^{-1-ps}|V|^p\,dX
\le s^{-p}\int_{\R^n\times(0,\infty)}w|\partial_yV|^p\,dX.
\end{equation}
The estimate first holds for smooth zero-trace functions and extends by
completion. Cutting off near $y=0$ has gradient error controlled by the
left-hand side of \eqref{eqn:zero-trace-hardy-p} on a shrinking strip.
A radial cutoff has error controlled by its tail, since $y\le2R$ on the
cutoff annulus. Smooth approximation away from the plane then gives
approximation in energy by compactly supported zero-trace functions.

Let $\delta_0$ be the Dirac mass at zero and let
$\delta_O(X)=\dist(X,\partial O)$.
For the measure-data applications below, let $\mu$ be a signed locally
finite Radon measure on $D$. The following observation transfers its
interior growth condition to the extension domain. Suppose that
\begin{equation}\label{eqn:plane-morrey-p}
|\mu|(B(x,r))\le Mr^{n-ps+p-\lambda}
\quad(x\in D,\ 0<r<\delta(x)/2),\qquad 0\le\lambda<p.
\end{equation}
It implies
\begin{equation}\label{eqn:lifted-morrey-p}
(|\mu|\otimes\delta_0)(B(X,r))
\le CMr^{-\lambda}w(B(X,r))
\quad(X\in O,\ 0<r<\delta_O(X)/2).
\end{equation}
Indeed, a ball meeting the plane has weighted volume comparable to
$r^{n+p-ps}$. Every point of its intersection with the plane has
boundary distance greater than $r$. A finite covering by smaller
interior balls therefore permits the use of \eqref{eqn:plane-morrey-p}.
On the plane, $\delta_O(x,0)=\delta(x)$.
The same argument applies to all balls if the growth condition on $\mu$
is imposed for every center and radius.

\begin{theorem}\label{thm:fractional-two-weight}
Let $\varnothing\ne\Gamma\subset\complement D$ be closed and assume
\begin{equation}\label{eqn:general-fractional-density}
\capacity_{s,p}(\complement D\cap\cl{B(\xi,r)},B(\xi,2r))
\ge\gamma_s r^{n-ps}
\quad(\xi\in\Gamma,\ r>0),
\end{equation}
where $\gamma_s>0$.
Let $p-ps\le\lambda<p$ and let $\mu\ge0$ be a Radon measure on $\R^n$
satisfying
\begin{equation}\label{eqn:fractional-global-growth}
\mu(B(x,r))\le M_{\mathrm g}r^{n+p-ps-\lambda}
\quad(x\in\R^n,\ r>0).
\end{equation}
Then
\begin{equation}\label{eqn:fractional-two-weight}
\int_D|u|^p\delta_\Gamma^{\lambda-p}\,d\mu
\le CM_{\mathrm g}[u]_{s,p}^p,
\end{equation}
for all $u\in C_c^\infty(D)$,
where $C$ depends only on $n,p,s,\lambda$ and $\gamma_s$.
\end{theorem}
\begin{proof}
For compact $K\subset\cl{B(\xi,r)}$, applying
\eqref{eqn:homogeneous-trace} to capacity competitors gives
\begin{equation}\label{eqn:fractional-capacity-lift}
\capacity_{s,p}(K,B(\xi,2r))
\le C\capacity_{p,w}(K\times\{0\},B((\xi,0),2r)).
\end{equation}
The weighted capacity of $\cl{B((\xi,0),r)}$ relative to
$B((\xi,0),2r)$ is bounded above by $Cr^{n-ps}$.
Thus \eqref{eqn:general-fractional-density} and
\eqref{eqn:fractional-capacity-lift} imply \eqref{eqn:cdc}
for $O$ on $\Gamma\times\{0\}$.
By \eqref{eqn:fractional-global-growth}, the measure
$(\mathbf1_D\mu)\otimes\delta_0$ satisfies \eqref{eqn:global-growth}
in the weighted space $\R^{n+1}$ with coefficient $CM_{\mathrm g}$.
Theorem \ref{thm:two-weight} gives
\begin{equation}\label{eqn:extension-two-weight}
\int_D|V(x,0)|^p\delta_\Gamma^{\lambda-p}\,d\mu
\le CM_{\mathrm g}\int_Ow|\nabla V|^p\,dX,
\end{equation}
for all $V\in C_c^\infty(O)$.
For $u\in C_c^\infty(D)$, substitute $V=\eta_REu$ in
\eqref{eqn:extension-two-weight}.
Equations \eqref{eqn:extension-p} and \eqref{eqn:extension-cutoff}
allow us to let $R\to\infty$ and prove \eqref{eqn:fractional-two-weight}.
\end{proof}

\begin{corollary}\label{cor:fractional-hardy}
Under \eqref{eqn:general-fractional-density},
\begin{equation}\label{eqn:hardy-fractional-p}
\int_D|u|^p\delta_\Gamma^{-ps}\,dx\le C[u]_{s,p}^p,
\end{equation}
for all $u\in C_c^\infty(D)$.
\end{corollary}
\begin{proof}
Apply \eqref{eqn:fractional-two-weight} with $d\mu=dx$ and $\lambda=p-ps$.
\end{proof}
The proof of Theorem \ref{thm:fractional-two-weight} also proves
\eqref{eqn:fractional-two-weight} and \eqref{eqn:hardy-fractional-p}
whenever $O$ satisfies weighted capacity density on $\Gamma\times\{0\}$.
In particular, $\Gamma=\complement D$ gives the boundary distance on $D$.
The regional inequalities in \cite{ChenSong2003,Dyda2004} use energy over
$D\times D$, whereas \eqref{eqn:hardy-fractional-p} uses the full-space
seminorm. For $n\ge2$ and $ps<n$, the latter inequality under fractional
capacity density is proved in \cite[Theorem 6.5]{ILTVHardy}.
For further relations between fractional capacity density and Hardy
inequalities, see \cite{IMV2024}.
Theorem \ref{thm:fractional-two-weight} includes general measure weights.
The characterization in \cite[Theorem 1.1]{BianchiBrascoSkZagati}
uses positive supersolutions for the integral fractional $p$-Laplacian;
here the supersolutions are for the weighted local extension equation.

\subsection{Subcritical traces and compactness}\label{sec:embeddings}
Let $\mathcal X$ be the completion of $C_c^\infty(D)$ in $[\cdot]_{s,p}$.
For bounded $D$, the fractional Poincar\'e inequality identifies this space
with $\cl{C_c^\infty(D)}^{W^{s,p}(\R^n)}$.
For unbounded $D$, we assume \eqref{eqn:hardy-fractional-p} with
$\Gamma=\complement D$ whenever $\mathcal X$ is used below.
This estimate controls the $L^p$ norm on every compact subset of $D$.
Elements of the completion therefore have locally integrable
representatives, extended by zero outside $D$.
Their difference quotients form a closed subspace of
$L^p(\R^n\times\R^n)$, so $\mathcal X$ is reflexive.

For $\eta\in C_c^\infty(\R^n)$ and $u\in C_c^\infty(D)$,
\begin{equation}\label{eqn:fractional-multiplication}
\begin{aligned}
{}[\eta u]_{s,p}^p\le{}&C\|\eta\|_\infty^p[u]_{s,p}^p\\
&+C\int_D |u(z)|^p
\left(\int_{\R^n}\frac{|\eta(x)-\eta(z)|^p}{|x-z|^{n+ps}}\,dx\right)dz.
\end{aligned}
\end{equation}
The inner integral in \eqref{eqn:fractional-multiplication} is bounded and is $O(|z|^{-n-ps})$ at infinity.
For unbounded $D$, fix $\xi\in\complement D$ and use
$\delta(z)\le|z-\xi|$. The inner integral and $|\eta(z)|^p$ are then
bounded by $C_\eta\delta(z)^{-ps}$ on $D$.
Equation \eqref{eqn:hardy-fractional-p}, or the fractional Poincar\'e
inequality for bounded $D$, gives
\begin{equation}\label{eqn:fractional-cutoff-bound}
\|\eta u\|_{W^{s,p}(\R^n)}\le C_\eta[u]_{s,p}.
\end{equation}
Applying this estimate to differences of smooth approximations shows
that multiplication by $\eta$ maps $\mathcal X$ continuously into both
$\mathcal X$ and $W^{s,p}(\R^n)$.

\begin{proposition}\label{prop:compact-fractional}
Let $\nu\ge0$ be locally finite on $D$ with bounded support.
Suppose that $U\in H^{1,p}_{\loc}(O;w)\cap C(\cl O)$ satisfies
\[
U\ge0\quad\text{in }O,\qquad U=0\quad\text{on }\partial O,
\qquad -\laplacian_{p,w}U\ge\nu\otimes\delta_0\quad\text{in }O.
\]
Then the identity on $C_c^\infty(D)$ extends to a compact trace
$T_\nu:\mathcal X\to L^p(D;\nu)$.
If $\eta\in C_c^\infty(\R^n)$ is equal to one near $\spt\nu$, then
for every $\epsilon>0$ there exists $C_{\epsilon,\eta}>0$ such that
\begin{equation}\label{eqn:compact-fractional}
\|T_\nu u\|_{L^p(\nu)}^p
\le\epsilon[u]_{s,p}^p+C_{\epsilon,\eta}\|\eta u\|_{L^p(\R^n)}^p,
\end{equation}
for all $u\in\mathcal X$.
\end{proposition}
\begin{proof}
Let $v\in C_c^\infty(D)$. Picone's inequality for $U+t$ with test trace
$v$, obtained by using $\eta_REv$ and passing to the limit by
\eqref{eqn:extension-cutoff}, gives
\begin{equation}\label{eqn:fractional-boundary-layer}
\int_{D\cap\{\delta<h\}}|v|^p\,d\nu
\le C\left(\sup_{\spt\nu\cap\{\delta<h\}}U(\cdot,0)\right)^{p-1}[v]_{s,p}^p.
\end{equation}
The coefficient tends to zero as $h\downarrow0$.
On the remaining compact part of $\spt\nu\times\{0\}$,
apply \eqref{eqn:picone}, with $S=U$, to cutoffs of $Ev$ in balls
compactly contained in $O$.
Uniform continuity of $U$ makes their oscillation factors arbitrarily small.
The first estimate in \eqref{eqn:extension-p} controls the gradient terms;
the second controls the cutoff terms by $C_r\|v\|_{L^p(\R^n)}^p$.
Summing over a finite covering with bounded overlap and using
\eqref{eqn:fractional-boundary-layer}, we obtain
\[
\int_D|v|^p\,d\nu\le\epsilon[v]_{s,p}^p
+C_\epsilon\|v\|_{L^p(\R^n)}^p.
\]
Apply this estimate to $v=\eta u$.
Since $\eta=1$ near $\spt\nu$, \eqref{eqn:fractional-cutoff-bound}
proves \eqref{eqn:compact-fractional} for smooth $u$, after adjusting
$\epsilon$. The estimate gives a continuous trace and extends by density
to $\mathcal X$.
For a bounded sequence $(u_j)$ in $\mathcal X$,
\eqref{eqn:fractional-cutoff-bound} and Rellich compactness give an
$L^p$-convergent subsequence of $(\eta u_j)$.
Equation \eqref{eqn:compact-fractional}, applied to differences, proves
that the corresponding traces converge in $L^p(\nu)$.
\end{proof}

Applying the supersolution construction to measures with interior growth
gives the following quantitative estimates and compactness criterion.
\begin{corollary}\label{cor:compact-fractional-distance}
Assume weighted capacity density on $\partial O$.
Let $p-ps\le\lambda<p$, $M>0$, and let $\nu\ge0$ be locally finite on $D$,
with
\begin{equation}\label{eqn:fractional-interior-growth}
\nu(B(x,r))\le Mr^{n+p-ps-\lambda}
\quad(x\in D,\ 0<r<\delta(x)/2).
\end{equation}
For $0\le\vartheta<p-\lambda$, the identity on $C_c^\infty(D)$ extends
to a continuous trace from $\cl{C_c^\infty(D)}^{W^{s,p}(\R^n)}$ into
$L^p(D;\delta^{-\vartheta}d\nu)$.
If $\|\delta\|_{L^\infty(\spt\nu)}<\infty$, then
\begin{equation}\label{eqn:frac-depth}
\int_D|T_\nu u|^p\delta^{-\vartheta}\,d\nu
\le CM\|\delta\|_{L^\infty(\spt\nu)}^{p-\lambda-\vartheta}[u]_{s,p}^p.
\end{equation}
For every $\epsilon>0$,
\begin{equation}\label{eqn:frac-epsilon}
\int_D|T_\nu u|^p\delta^{-\vartheta}\,d\nu
\le\epsilon[u]_{s,p}^p+
CM^{ps/(p-\lambda-\vartheta)}
\epsilon^{-(ps-p+\lambda+\vartheta)/(p-\lambda-\vartheta)}\|u\|_p^p.
\end{equation}
Both estimates hold on the stated Sobolev space.
The trace is compact if $\spt\nu$ is bounded.
The constants depend only on $n,p,s,\lambda,\vartheta$ and the
weighted capacity density constant.
\end{corollary}
\begin{proof}
Fix $u\in C_c^\infty(D)$.
The assertion is immediate for $\nu=0$. Otherwise,
$\delta^{-\vartheta}\nu$ satisfies \eqref{eqn:fractional-interior-growth}
with $\lambda$ replaced by $\lambda+\vartheta<p$.
By \eqref{eqn:lifted-morrey-p}, Proposition \ref{prop:barrier} applies
to $(\delta^{-\vartheta}\nu)\otimes\delta_0$ when its support has finite
depth. Picone's inequality for the resulting supersolution, applied to
$\eta_REu$ and followed by $R\to\infty$, gives \eqref{eqn:frac-depth}
by \eqref{eqn:extension-p} and \eqref{eqn:extension-cutoff}.
The same argument for
$(\mathbf1_{\{\delta<h\}}\delta^{-\vartheta}\nu)\otimes\delta_0$ gives
\begin{equation}\label{eqn:frac-subcritical-boundary-layer}
\int_{\{\delta<h\}}|u|^p\delta^{-\vartheta}\,d\nu
\le CMh^{p-\lambda-\vartheta}[u]_{s,p}^p,
\end{equation}
for all $u\in C_c^\infty(D)$.
To estimate the interior contribution, take $32R<\delta(x)$.
By \eqref{eqn:fractional-interior-growth}, the measure
$(\mathbf1_{B(x,R)}\nu)\otimes\delta_0$ satisfies \eqref{eqn:full}
at scale $R$ with the same exponent $\lambda$ and growth coefficient $CM$.
For small balls this follows by recentering and the interior growth
condition; for larger balls it follows from
$\nu(B(x,R))\le MR^{n+p-ps-\lambda}$.
Apply \eqref{eqn:trace} to $Eu$ multiplied by a cutoff equal to one
on $B_{n+1}((x,0),R)$ and supported in $B_{n+1}((x,0),2R)$.
On this support the convolution uses only values of $u$ in $B(x,4R)$.
Equations \eqref{eqn:extension-derivative} and \eqref{eqn:extension-p}
therefore give
\begin{equation}\label{eqn:fractional-local-trace}
\begin{aligned}
\int_{B(x,R)}|u|^p\,d\nu\le CM\bigg(&R^{p-\lambda}
\iint_{B(x,4R)\times B(x,4R)}
\frac{|u(z)-u(z')|^p}{|z-z'|^{n+ps}}\,dz\,dz'\\
&+R^{p-ps-\lambda}\int_{B(x,4R)}|u|^p\,dx\bigg).
\end{aligned}
\end{equation}
Cover $\{\delta\ge h\}$ by balls of radius $h/128$, with bounded
overlap of the fourfold balls. On these balls
$\delta^{-\vartheta}\le Ch^{-\vartheta}$.
Summing \eqref{eqn:fractional-local-trace} and using
\eqref{eqn:frac-subcritical-boundary-layer} yields
\[
\int_D|u|^p\delta^{-\vartheta}\,d\nu
\le CM\left(h^{p-\lambda-\vartheta}[u]_{s,p}^p
+h^{p-ps-\lambda-\vartheta}\|u\|_p^p\right).
\]
Choosing $h$ so that $CMh^{p-\lambda-\vartheta}=\epsilon$ proves
\eqref{eqn:frac-epsilon}. Density defines $T_\nu$ and extends the estimates
to the stated Sobolev space.
For bounded support, the supersolution for
$(\delta^{-\vartheta}\nu)\otimes\delta_0$ is continuous and vanishes
on $\partial O$. Proposition \ref{prop:compact-fractional} gives
compactness on $\mathcal X$, hence on the stated Sobolev space.
\end{proof}

\subsection{The critical trace inequality}
\begin{corollary}\label{cor:fractional-distance}
Assume weighted capacity density on $\partial O$.
Let $p-ps\le\lambda<p$, $M>0$, and let $\nu\ge0$ be locally finite on $D$,
satisfying \eqref{eqn:fractional-interior-growth}.
The identity on $C_c^\infty(D)$ extends to a continuous trace
from $\cl{C_c^\infty(D)}^{W^{s,p}(\R^n)}$ into
$L^p(D;\delta^{\lambda-p}d\nu)$, with
\begin{equation}\label{eqn:frac-critical}
\int_D|T_\nu u|^p\delta^{\lambda-p}\,d\nu
\le CM[u]_{s,p}^p.
\end{equation}
The constant depends only on $n,p,s,\lambda$ and the
weighted capacity density constant.
\end{corollary}
\begin{proof}
Let $u\in C_c^\infty(D)$. Multiplying
\eqref{eqn:fractional-local-trace} by $R^{\lambda-p}$ and summing over
a Whitney covering with bounded overlap of the fourfold balls gives
\[
\int_D|u|^p\delta^{\lambda-p}\,d\nu
\le CM\left([u]_{s,p}^p+\int_D|u|^p\delta^{-ps}\,dx\right).
\]
The observation after Corollary \ref{cor:fractional-hardy} permits
\eqref{eqn:hardy-fractional-p} under the present geometric hypothesis.
It bounds the last integral and proves \eqref{eqn:frac-critical}.
Density defines the trace on the stated Sobolev space; it agrees locally
$\nu$-almost everywhere with the traces in
Corollary \ref{cor:compact-fractional-distance}.
\end{proof}

\section{The fractional Dirichlet problems}\label{sec:fractional}
Let $D\subset\R^n$ be bounded and open, $n\ge2$, and $0<s<1$.
We consider the restricted fractional Laplacian $(-\laplacian)^s$ with
zero exterior values and the spectral power $L_D^s$ of the Dirichlet
Laplacian. Both powers are defined by spectral calculus for nonnegative
self-adjoint operators: $-\Delta$ on $L^2(\R^n)$ and $L_D$ on $L^2(D)$,
where $L_D$ is associated with the closed form
$\int_D|\nabla v|^2\,dx$ on $H^{1,2}_0(D)$.
Thus the restricted operator acts on the zero extension to $\R^n$,
whereas the spectral operator is the power of $L_D$.
The boundedness of $D$ is used for the identification with
Green potentials. The extension problem on an unbounded base is treated
in Remark \ref{rem:unbounded-extension}.
For the restricted fractional Laplacian, sharp boundary regularity for
bounded data is studied in \cite{RosOtonSerra2014}; Morrey-space data and
Schr\"odinger potentials are treated in \cite{Fall2019}.
For spectral powers, see \cite{CaffarelliStinga2016,Grubb2016}, and
\cite{AbatangeloDupaigne2017} for measure data and nonhomogeneous boundary traces.

We assume that the signed locally finite measure $\mu$ on $D$ satisfies
\begin{equation}\label{eqn:frac-morrey}
|\mu|(B(x,r))\le Mr^{n-2s+\beta}
\quad(x\in D,\ 0<r<\delta(x)/2),\qquad 0<\beta\le2s.
\end{equation}
Unlike the Morrey condition in \cite{Fall2019}, \eqref{eqn:frac-morrey}
does not require a growth bound on balls meeting the boundary.
In particular, it permits measures with infinite mass near the boundary.
Combining boundary estimates under capacity density with standard interior
estimates, we obtain boundary vanishing and global H\"older continuity.
For nonnegative $\mu$, its Green potential is the increasing limit of
the energy solutions with data $\mathbf1_{\{\delta>1/j\}}\mu$.
For signed $\mu$, it is the difference of the potentials of $\mu^+$ and
$\mu^-$ when both are finite. The theorem below proves this finiteness.

\subsection{Extension and capacity}
The geometric assumptions are
\begin{equation}\label{eqn:frac-cdc}
\capacity_{s,2}(\cl{B(\xi,r)}\setminus D,B(\xi,2r))
\ge\gamma_s r^{n-2s}\quad(\xi\in\partial D,\ r>0)
\end{equation}
for the restricted operator, and
\begin{equation}\label{eqn:ordinary-cdc}
\capacity_2(\cl{B(\xi,r)}\setminus D,B(\xi,2r))
\ge\gamma_2 r^{n-2}\quad(\xi\in\partial D,\ r>0)
\end{equation}
for the spectral operator. Here $\gamma_s,\gamma_2>0$, and
$\capacity_2$ denotes relative variational $2$-capacity.
Set $w(x,y)=|y|^{1-2s}$.
The extension domain $O$ is $\R^{n+1}\setminus(\complement D\times\{0\})$
for the restricted operator and $D\times\R$ for the spectral operator.
The even extension equation, with the normalizations of
\cite{CS2007,ST2010}, is
\begin{equation}\label{eqn:extension}
\begin{cases}
-\divergence(w\nabla U)=c_s\mu\otimes\delta_0&\text{in }O,\\
U=0&\text{on }\partial O,
\end{cases}
\qquad
c_s=2^{2-2s}\dfrac{\Gamma(1-s)}{\Gamma(s)}.
\end{equation}
For either domain, $\delta_O(x,0)=\delta(x)$.
Since
\begin{equation}\label{eqn:linear-ball-volume}
w(B((x,y),r))\asymp r^{n+1}(|y|+r)^{1-2s},
\end{equation}
with comparison constants depending only on $n,s$,
\eqref{eqn:frac-morrey} is equivalent, up to constants depending on
$n,s,\beta$, to \eqref{eqn:morrey} for $\mu\otimes\delta_0$ on $O$
with $p=2$ and $\lambda=2-\beta$.
The forward implication follows from the covering argument proving
\eqref{eqn:lifted-morrey-p}; the reverse follows by taking centers on the plane.

For compact $E\subset\cl{B(\xi,r)}$, the condenser comparison
\cite[Lemma 2.2 and (2.8)]{Bjorn2024} gives
\begin{equation}\label{eqn:slit-capacity}
\begin{aligned}
C^{-1}\capacity_{s,2}(E,B(\xi,2r))
&\le\capacity_{2,w}(E\times\{0\},B((\xi,0),2r))\\
&\le C\capacity_{s,2}(E,B(\xi,2r)).
\end{aligned}
\end{equation}
Together with \eqref{eqn:linear-ball-volume}, this shows that
\eqref{eqn:frac-cdc} is equivalent to weighted capacity density for the slit domain.
To include points over the interior of $\complement D$, either a disk
of comparable radius lies in $\complement D$, or a boundary point is
within a fixed fraction of that radius. In the latter case, recentering
and fixed-dilation capacity comparison give the same lower bound.

For the cylinder, slicing capacity competitors horizontally gives
\begin{equation}\label{eqn:cylinder-capacity}
\begin{aligned}
&\capacity_{2,w}\bigl(\cl{B((\xi,y_0),r)}\setminus(D\times\R),
 B((\xi,y_0),2r)\bigr)\\
&\hspace{2em}\ge c\gamma_2r^{n-2}r(|y_0|+r)^{1-2s},
\end{aligned}
\end{equation}
for all $\xi\in\partial D$, $y_0\in\R$ and $r>0$.
Indeed, for $|y-y_0|<r/4$, the removed horizontal section contains
$\cl{B(\xi,r/2)}\setminus D$; integrating its capacity bound gives
\eqref{eqn:cylinder-capacity}.
For the converse, use a ball centered at $(\xi,4r)$ and a product
competitor $\varphi(x)\eta(y)$, where $\varphi$ is a capacity competitor
supported in $B(\xi,3r/2)$, and $\eta=1$ for $|y-4r|\le r$ and
vanishes for $|y-4r|\ge5r/4$.
Its energy is bounded by
$Cr^{2-2s}\int_{B(\xi,3r/2)}|\nabla\varphi|^2\,dx$,
using the Poincar\'e inequality for the term containing $\eta'$.
Taking infima and comparing fixed dilations proves \eqref{eqn:ordinary-cdc}.
Thus Theorem \ref{thm:local} applies to \eqref{eqn:extension} under the
corresponding capacity condition.

\begin{theorem}\label{thm:fractional}
Let $0<\beta\le2s$ and suppose that \eqref{eqn:frac-morrey} holds.
Assume \eqref{eqn:frac-cdc} for the restricted operator, or
\eqref{eqn:ordinary-cdc} for the spectral operator.
Then the Green potential $u$ is finite, belongs to $C(\cl D)$,
and vanishes on $\partial D$.
It has a bounded even extension solving \eqref{eqn:extension}, and
\begin{equation}\label{eqn:frac-estimate}
\begin{aligned}
|u(x)|&\le CM\|\delta\|_{L^\infty(\spt\mu)}^{\beta-\alpha_{\mathrm{ss}}}
\min\{\delta(x),\|\delta\|_{L^\infty(\spt\mu)}\}^{\alpha_{\mathrm{ss}}},\\
[u]_{C^{\alpha_{\mathrm{gh}}}(\cl D)}
&\le CM\|\delta\|_{L^\infty(\spt\mu)}^{\beta-\alpha_{\mathrm{gh}}},
\end{aligned}
\end{equation}
for all $x\in D$.
Here $\alpha_{\mathrm{ss}}$ and $\alpha_{\mathrm{gh}}$ are the exponents
furnished by Theorem \ref{thm:local} on $O$.
They and $C$ depend only on $n,s,\beta$ and the corresponding capacity
density constant. For the restricted operator, the H\"older estimate
holds on $\R^n$ after zero extension, and
$(-\laplacian)^su=\mu$ distributionally in $D$.
For the spectral operator,
\begin{equation}\label{eqn:spectral-distribution}
\int_DuL_D^s\varphi\,dx=\int_D\varphi\,d\mu,
\end{equation}
for all $\varphi\in C_c^\infty(D)$.
\end{theorem}
\begin{proof}
We first describe the energy solutions for compactly supported data.
For a bounded open $G\subset D$, the slit energy space is the completion
of $C_c^\infty(\R^{n+1}\setminus(\complement G\times\{0\}))$
in the norm $(\int_{\R^{n+1}}w|\nabla V|^2\,dX)^{1/2}$.
By \eqref{eqn:extension-p} and \eqref{eqn:homogeneous-trace}, its trace
space is the closure of $C_c^\infty(G)$ in $[\cdot]_{s,2}$.
The fractional Poincar\'e inequality controls the $L^2(G)$ norm of the trace.
For the cylinder, the energy space is the completion in the same norm
of finite sums $\varphi(x)\eta(y)$ with
$\varphi\in C_c^\infty(G)$ and $\eta\in C_c^\infty(\R)$.
The ordinary Poincar\'e inequality on each horizontal section gives
\begin{equation}\label{eqn:cylinder-poincare}
\int_{G\times\R}w|V|^2\,dX
\le\lambda_1(G)^{-1}\int_{G\times\R}w|\nabla V|^2\,dX,
\end{equation}
where $\lambda_1(G)>0$ is the first Dirichlet eigenvalue.
Functions in either space are extended by zero outside their domains.

Suppose that $\mu$ is compactly supported in $G$ and satisfies
\eqref{eqn:frac-morrey}. A finite covering of its support and
\eqref{eqn:trace} show that
$V\mapsto\int_GV(x,0)\,d\mu$ is continuous in the energy norm.
The local $L^2(w)$ term is controlled by
\eqref{eqn:cylinder-poincare} in the cylinder and the homogeneous
weighted Sobolev inequality in the slit.
Thus there is a unique energy solution of
\begin{equation}\label{eqn:linear-energy-equation}
\int_{\R^{n+1}}w\nabla U\cdot\nabla V\,dX
=c_s\int_GV(x,0)\,d\mu
\end{equation}
for every $V$ in the corresponding energy space.
Reflection in $y=0$ and uniqueness imply that $U$ is even.
For the rest of this proof, take $G=D$.

In the slit case, the extension of its trace $u$ is the
Caffarelli--Silvestre extension. For a bounded continuous trace supported
in $\cl D$, its formula is
\begin{equation}\label{eqn:poisson}
\begin{aligned}
U(x,y)&=P^s_{|y|}*u(x)\quad(y\ne0),\\
P^s_y(x)&=\frac{\Gamma(n/2+s)}{\pi^{n/2}\Gamma(s)}
\frac{y^{2s}}{(|x|^2+y^2)^{(n+2s)/2}}\quad(y>0).
\end{aligned}
\end{equation}
The kernels are weighted harmonic and form an approximate identity
\cite[Section 2.4]{CS2007}.
A bounded weighted harmonic extension in the upper half-space is
uniquely determined by its continuous trace: the difference of two such
extensions has zero trace, and its odd reflection is a bounded entire
weighted harmonic function. Local finite energy across the plane follows
by testing with the positive and negative $\epsilon$-truncations and
letting $\epsilon\downarrow0$.
The interior oscillation estimate on expanding balls makes the reflection
constant, and oddness makes it zero. Thus \eqref{eqn:poisson} also holds
for the bounded extensions constructed below.

For $\varphi\in C_c^\infty(D)$, take a smooth compactly supported
upper-half-space function $\Phi$ with trace $\varphi$ and support over $D$.
Equation \eqref{eqn:extension} and evenness give
\begin{equation}\label{eqn:linear-conormal}
\int_{\R^n\times(0,\infty)}y^{1-2s}\nabla U\cdot\nabla\Phi\,dX
=\frac{c_s}{2}\int_D\varphi\,d\mu.
\end{equation}
The left-hand side depends only on the trace of $\Phi$.
The conormal identity in \cite[Section 3]{CS2007}, applied to
\eqref{eqn:poisson}, identifies it with
$\frac{c_s}{2}\int_{\R^n}u(-\laplacian)^s\varphi\,dx$.
For bounded traces in the local Sobolev class, this identity follows by
horizontal convolution: the extensions converge in weighted $H^1$ on
compact cylinders over $D$, while the fractional pairings converge
because $(-\laplacian)^s\varphi\in L^1(\R^n)$.
Consequently $(-\laplacian)^su=\mu$ in $D$.

For the cylinder, let $(\lambda_k,e_k)$ be an orthonormal Dirichlet
eigenbasis on $D$. The scalar extension profile in \cite{ST2010} is
$\psi_s(t)=2^{1-s}t^sK_s(t)/\Gamma(s)$, where $K_s$ is the modified
Bessel function. It satisfies
\[
\psi_s(0)=1,\qquad
-\lim_{t\downarrow0}t^{1-2s}\psi_s'(t)=c_s/2,
\qquad
\int_0^\infty t^{1-2s}(|\psi_s'|^2+|\psi_s|^2)\,dt=c_s/2.
\]
Applying these identities to finite eigenfunction sums and then passing to
the energy completion gives
\begin{equation}\label{eqn:spectral-energy}
\begin{aligned}
E_sv(x,y)&=\sum_k(v,e_k)_{L^2(D)}e_k(x)\psi_s(\sqrt{\lambda_k}|y|),\\
\inf_{\trace V=v}\int_{D\times\R}w|\nabla V|^2\,dX
&=\int_{D\times\R}w|\nabla E_sv|^2\,dX
=c_s\|L_D^{s/2}v\|_2^2.
\end{aligned}
\end{equation}
Thus the trace space is $\operatorname{Dom}(L_D^{s/2})$.
For compactly supported $\mu$, the right-hand side of
\eqref{eqn:linear-energy-equation} depends only on the trace.
For zero-trace functions this follows from
\eqref{eqn:zero-trace-hardy-p}, which supplies approximation by energy
functions vanishing near the plane.
The solution therefore minimizes energy among extensions of its trace,
so $U=E_su$. Testing \eqref{eqn:linear-energy-equation} with $E_sv$ gives
\begin{equation}\label{eqn:spectral-weak}
(L_D^{s/2}u,L_D^{s/2}v)_{L^2(D)}=\langle\mu,v\rangle,
\end{equation}
for all $v\in\operatorname{Dom}(L_D^{s/2})$.
This identifies the trace with the spectral energy solution.

Now let $\mu\ge0$ satisfy \eqref{eqn:frac-morrey}, without compact
support in $D$, and let $U_j$ be the energy solution for
$\mathbf1_{\{\delta>1/j\}}\mu$.
Proposition \ref{prop:barrier} gives a bounded continuous supersolution
$S$ for $c_s\mu\otimes\delta_0$ on $O$.
The energy solutions on bounded truncations of $O$ are bounded above
by $S$, by comparison. Their energy spaces increase to a dense subspace
of the full energy space, so these solutions converge strongly in energy
to $U_j$. Consequently
\begin{equation}\label{eqn:linear-monotone-approximation}
0\le U_j\le U_{j+1}\le S,
\end{equation}
in $O$.
Positivity and the middle inequality follow by testing the negative part
and by comparison in \eqref{eqn:linear-energy-equation}.
The local energy estimate \eqref{eqn:local-approximation-energy}
and interior regularity now give local uniform convergence to a bounded
even function $U$, and weak convergence in $H^{1,2}_{\loc}(O;w)$.
Passing to the limit in \eqref{eqn:linear-energy-equation} for compactly
supported tests proves \eqref{eqn:extension}.
The bound by $S$ gives continuous zero boundary values.

The trace $u=U(\cdot,0)$ is the increasing limit defining the Green
potential. For signed $\mu$, apply the construction to $\mu^+$ and $\mu^-$
and subtract; both limits are finite by \eqref{eqn:linear-monotone-approximation}.
The pointwise estimate and the two-point argument of Theorem
\ref{thm:local} prove \eqref{eqn:frac-estimate}.
For the restricted operator, \eqref{eqn:poisson} and
\eqref{eqn:linear-conormal} give the distributional equation for the limit.
For the spectral operator, $\varphi\in C_c^\infty(D)$ belongs to
$\operatorname{Dom}(L_D)$, so $L_D^s\varphi\in L^2(D)$.
The common bound in \eqref{eqn:linear-monotone-approximation} gives
convergence of the traces in $L^2(D)$.
Passing to the limit in \eqref{eqn:spectral-weak} proves
\eqref{eqn:spectral-distribution}.
\end{proof}

\begin{remark}\label{rem:unbounded-extension}
Let $D\subsetneq\R^n$ be any nonempty open set.
Assume \eqref{eqn:frac-morrey} and
$\|\delta\|_{L^\infty(\spt\mu)}<\infty$.
Under the corresponding capacity condition,
Theorems \ref{thm:local} and \ref{thm:local-uniqueness} give a unique
bounded weak solution of \eqref{eqn:extension} with continuous zero
boundary values, and its trace satisfies \eqref{eqn:frac-estimate}.
The Green-potential identification in Theorem \ref{thm:fractional}
uses boundedness of $D$.

A partial-boundary version of the pointwise estimate also follows.
Let $\varnothing\ne\Gamma\subset\partial D$ be closed, impose the
corresponding capacity condition only at $\xi\in\Gamma$, and replace
$\delta$ by $\delta_\Gamma$ in \eqref{eqn:frac-morrey} and the support-depth
assumption, evaluating $|\mu|$ on $B(x,r)\cap D$.
Proposition \ref{prop:barrier} applies with boundary set
$\Gamma\times\{0\}$ in the slit and $\Gamma\times\R$ in the cylinder.
It gives a bounded exhaustion solution whose trace satisfies the first
estimate in \eqref{eqn:frac-estimate}, with $\delta_\Gamma$ throughout.
For bounded $D$, the energy approximation used above identifies this
trace with the Green potential.
\end{remark}

\begin{remark}\label{rem:smooth-comparison}
On a bounded $C^{1,1}$ domain, the two operators have different boundary
behavior. For a fixed pole $z\in D$ and $x$ near $\partial D$, their
Green functions satisfy
\[
G_{(-\laplacian)^s,D}(x,z)\asymp\delta(x)^s,
\qquad G_{L_D^s}(x,z)\asymp\delta(x),
\]
with constants allowed to depend on $z$; see
\cite[Corollary 1.2]{ChenKimSong2010} and
\cite[Theorem 3.1]{SongVondracek2004}.
For the spectral datum $f=\delta^{-\sigma}$, $0\le\sigma<2$,
\cite[Theorem 3.4]{AbatangeloGomezCastroVazquez2023} gives
\[
L_D^{-s}f(x)\asymp
\begin{cases}
\delta(x)^{2s-\sigma},&2s-\sigma<1,\\
\delta(x)\log(e\diam(D)/\delta(x)),&2s-\sigma=1,\\
\delta(x),&2s-\sigma>1.
\end{cases}
\]
Thus continuous vanishing at the boundary holds precisely for
$\sigma<2s$. Condition \eqref{eqn:frac-morrey} for some $\beta>0$
has the same threshold: the quotient on an admissible ball is comparable
to $\delta(x)^{-\sigma}r^{2s-\beta}$, and one may take
$r=\delta(x)/4$ to test necessity.
\end{remark}

\subsection{Uniqueness}
\begin{theorem}\label{thm:fractional-uniqueness}
Under \eqref{eqn:frac-cdc} in the slit or \eqref{eqn:ordinary-cdc} in
the cylinder, the bounded weak solution of \eqref{eqn:extension} with
continuous zero boundary values is unique.
For bounded $D$, the restricted problem also has at most one
continuous distributional solution on $\R^n$ with zero exterior values.
\end{theorem}
\begin{proof}
The capacity comparisons \eqref{eqn:slit-capacity} and
\eqref{eqn:cylinder-capacity} permit the use of
Theorem \ref{thm:local-uniqueness} with $p=2$ on $O$.
This proves uniqueness of the bounded extension.
For the restricted problem, let $v\in C(\R^n)$ vanish on $\complement D$
and satisfy $(-\laplacian)^sv=0$ in $D$.
Convolution with a smooth approximate identity supported in $B(0,\epsilon)$
gives a smooth fractional harmonic function on $\{x\in D:\delta(x)>\epsilon\}$.
Uniform continuity of $v$ and its zero exterior values imply that the
supremum of the convolution in absolute value off this set tends to zero.
The maximum principle on this bounded open set gives the same bound
inside it. Letting $\epsilon\downarrow0$ yields $v=0$ in $\R^n$.
\end{proof}

\subsection{Lower bounds for capacity density}
\begin{theorem}\label{thm:fractional-capacity}
Fix $0<\beta\le2s$.
Suppose that the nonnegative Green potentials for either operator satisfy
the pointwise bound in \eqref{eqn:frac-estimate}, with fixed $C$ and
$\alpha_{\mathrm{ss}}>0$, for every $\xi\in\partial D$ and $r>0$, with datum
\[
r^{-2s}\mathbf1_{D\cap(B(\xi,r/2)\setminus\cl{B(\xi,r/4)})}\,dx.
\]
Then \eqref{eqn:frac-cdc} holds for the restricted operator, and
\eqref{eqn:ordinary-cdc} holds for the spectral operator.
The positive density constant depends only on $n,s,\beta,C$ and
$\alpha_{\mathrm{ss}}$.
\end{theorem}
\begin{proof}
The datum in the statement satisfies \eqref{eqn:frac-morrey} with
$M\le Cr^{-\beta}$ and support depth at most $r/2$.
Its potential and extension therefore satisfy
\begin{equation}\label{eqn:plane-bound}
U_r(x,0)\le C(\delta(x)/r)^{\alpha_{\mathrm{ss}}},
\end{equation}
for all $x\in D$.
By translation and scaling, it suffices to work with $\xi=0$ and $r=1$.
We use the slit or cylinder energy spaces from the proof of
Theorem \ref{thm:fractional}, with bases $B(0,1)$ and $D\cap B(0,1)$.
Their traces satisfy an $L^2$ Poincar\'e inequality even without a capacity
assumption on $D$: in the slit this is the fractional Poincar\'e inequality
for functions with bounded support; in the cylinder it follows from
\eqref{eqn:spectral-energy} and the spectral gap.
The bounded annular density therefore defines a continuous energy functional.
Its compact interior restrictions converge in the dual norm, since the
corresponding densities converge in $L^2$.
Thus $U_1$ is the full energy solution for the annular datum.

Let $Z$ and $V$ be the reference energy solutions, with zero boundary
values, satisfying
\begin{equation}\label{eqn:linear-capacity-problems}
\begin{aligned}
-\divergence(w\nabla Z)
&=c_s\mathbf1_{B(0,1/2)\setminus\cl{B(0,1/4)}}\,dx\otimes\delta_0
\\
-\divergence(w\nabla V)
&=c_s\mathbf1_{D\cap(B(0,1/2)\setminus\cl{B(0,1/4)})}\,dx\otimes\delta_0.
\end{aligned}
\end{equation}
The equation for $Z$ holds in $\R^{n+1}\setminus(\complement B(0,1)\times\{0\})$
in the slit case and in $B(0,1)\times\R$ in the cylinder case.
The equation for $V$ holds in
$\R^{n+1}\setminus(\complement(D\cap B(0,1))\times\{0\})$ or
$(D\cap B(0,1))\times\R$, respectively.
We extend both energy functions by zero outside these domains.
The zero-boundary space for $D\cap B(0,1)$ embeds into the spaces for
both $D$ and $B(0,1)$.
The positive part $(V-U_1)_+$ is an admissible test in that smaller
space. The right-hand sides agree on this test, so subtraction of
\eqref{eqn:linear-energy-equation} gives $0\le V\le U_1$.

The reference ball has the required capacity density.
Theorem \ref{thm:fractional}, local boundedness and the weak Harnack
inequality give
\begin{equation}\label{eqn:reference}
0\le Z\le C,
\qquad Z(x,0)\ge c>0,
\end{equation}
on the extension domain and for $x\in B(0,1/2)$, respectively.
For the lower bound, the nonzero source makes $Z$ nontrivial.
The homogeneous weak Harnack inequality for this nonnegative
supersolution, along a finite chain of interior balls, gives positivity
on $\cl{B(0,1/2)}\times\{0\}$.
The reference problem is fixed, so the constants depend only on $n,s$.

In the slit case, let $H$ be the weighted capacity potential of
$(\cl{B(0,1)}\setminus D)\times\{0\}$ in $B_{n+1}(0,2)$,
extended by zero. Its energy is
\begin{equation}\label{eqn:slit-removed-capacity}
\int_{\R^{n+1}}w|\nabla H|^2\,dX
=\capacity_{2,w}((\cl{B(0,1)}\setminus D)\times\{0\},B_{n+1}(0,2)).
\end{equation}
The trace and Poincar\'e inequalities on the fixed ball give
\[
\int_{B(0,1)}|H(x,0)|^2\,dx
\le C\int_{\R^{n+1}}w|\nabla H|^2\,dX.
\]
Since $Z$ is bounded by \eqref{eqn:reference},
$\int_{\R^{n+1}}wZ^2|\nabla H|^2\,dX\le C\int_{\R^{n+1}}w|\nabla H|^2\,dX$.
Moreover, $Z(1-H)$ belongs to the energy space for $D\cap B(0,1)$:
it vanishes quasi-everywhere on the added slit and on the original slit.
The zero-boundary capacity characterization and truncation at infinity
justify this assertion.

In the cylinder case, let $h$ be the variational capacity potential of
$\cl{B(0,1)}\setminus D$ in $B(0,2)$ and set $H(x,y)=h(x)$.
Then
\begin{equation}\label{eqn:cylinder-removed-capacity}
\int_{B(0,2)}|\nabla h|^2\,dx
=\capacity_2(\cl{B(0,1)}\setminus D,B(0,2)),
\qquad
\int_{B(0,2)}h^2\,dx\le C\int_{B(0,2)}|\nabla h|^2\,dx.
\end{equation}
To estimate the product with $Z$, let $e_1$ be the positive first
Dirichlet eigenfunction on $B(0,1)$.
The annular density is bounded by $Ce_1$.
Positivity of the spectral inverse and of its extension gives
\[
Z(x,y)\le Ce_1(x)\psi_s(\sqrt{\lambda_1}|y|),
\qquad
\sup_{x\in B(0,1)}\int_\R|y|^{1-2s}Z(x,y)^2\,dy\le C.
\]
The latter integral is finite because $\psi_s$ is bounded at zero and
decays exponentially at infinity. Hence
$\int_{\R^{n+1}}wZ^2|\nabla H|^2\,dX\le C\int_{B(0,2)}|\nabla h|^2\,dx$.
The product $Z(1-H)$ has zero lateral values on the cylinder over
$D\cap B(0,1)$ and belongs to its energy space.

In either case, testing the equation for $Z$ in
\eqref{eqn:linear-capacity-problems} with $ZH^2$ gives
\[
\begin{aligned}
\int_{\R^{n+1}}wH^2|\nabla Z|^2\,dX
\le{}&C\int_{\R^{n+1}}wZ^2|\nabla H|^2\,dX\\
&+C\int_{B(0,1/2)\setminus\cl{B(0,1/4)}}Z(x,0)H(x,0)^2\,dx.
\end{aligned}
\]
Energy approximation justifies the test; in the cylinder we first cut
off in $y$ and then pass to the limit using the preceding vertical bound.
Thus the energy of $ZH$ is bounded by the capacity in
\eqref{eqn:slit-removed-capacity} or \eqref{eqn:cylinder-removed-capacity}.
The variational equations for $Z$ and $V$, with competitor $Z(1-H)$
for $V$, give directly
\begin{equation}\label{eqn:linear-capacity-comparison}
\int_{\R^{n+1}}w|\nabla(V-Z)|^2\,dX\le\int_{\R^{n+1}}w|\nabla(ZH)|^2\,dX.
\end{equation}
By \eqref{eqn:linear-capacity-comparison}, the trace Poincar\'e inequality on the reference space yields
\begin{equation}\label{eqn:plane-stability}
\int_{B(0,1)}|V(x,0)-Z(x,0)|^2\,dx\le
\begin{cases}
C\displaystyle\int_{\R^{n+1}}w|\nabla H|^2\,dX,&\text{in the slit case},\\
C\displaystyle\int_{B(0,2)}|\nabla h|^2\,dx,&\text{in the cylinder case}.
\end{cases}
\end{equation}
Choose $0<\theta<1/2$ so that the bound in \eqref{eqn:plane-bound}
is less than half the lower bound in \eqref{eqn:reference} on
$D\cap B(0,\theta)$.
Since $V\le U_1$ and $V(x,0)=0$ off $D$, the left-hand side of
\eqref{eqn:plane-stability} is bounded below by $c\theta^n$.
Equations \eqref{eqn:slit-removed-capacity} and
\eqref{eqn:cylinder-removed-capacity} give positive lower bounds for the
respective capacities. Scaling back, and using \eqref{eqn:slit-capacity}
in the slit case, proves \eqref{eqn:frac-cdc} or \eqref{eqn:ordinary-cdc}.
\end{proof}

\subsection{Hardy inequalities for the two linear extensions}
For the restricted operator, \eqref{eqn:fractional-two-weight} with $p=2$
can be written using $\|(-\laplacian)^{s/2}u\|_2^2$, since this energy
is a fixed positive multiple of $[u]_{s,2}^2$.
For the spectral half Laplacian on weakly mean convex domains,
\cite{FilippasMoschiniTertikas2015} proves Hardy--Sobolev--Maz'ya inequalities
with the best Hardy constant.
The following estimate uses capacity density and allows measure weights.
\begin{corollary}\label{thm:linear-two-weight}
Let $D\subsetneq\R^n$ be nonempty and open, $0<s<1$, and
$\varnothing\ne\Gamma\subset\partial D$ be closed.
Let $2-2s\le\lambda<2$ and assume \eqref{eqn:fractional-global-growth}
with $p=2$, together with
\begin{equation}\label{eqn:spectral-partial-cdc}
\capacity_2(\complement D\cap\cl{B(\xi,r)},B(\xi,2r))
\ge\gamma_2r^{n-2}\quad(\xi\in\Gamma,\ r>0),
\end{equation}
where $\gamma_2>0$.
Then
\begin{equation}\label{eqn:spectral-two-weight}
\int_D|u|^2\delta_\Gamma^{\lambda-2}\,d\mu
\le CM_{\mathrm g}\|L_D^{s/2}u\|_2^2,
\end{equation}
for all $u\in C_c^\infty(D)$,
where $C$ depends only on $n,s,\lambda,\gamma_2$.
\end{corollary}
\begin{proof}
On $D\times\R$, use the boundary set $\Gamma\times\R$. The slicing argument for \eqref{eqn:cylinder-capacity},
applied to \eqref{eqn:spectral-partial-cdc}, gives weighted capacity
density on this boundary set.
Equation \eqref{eqn:fractional-global-growth} implies
\eqref{eqn:global-growth} for $(\mathbf1_D\mu)\otimes\delta_0$.
Since the distance from $(x,y)$ to $\Gamma\times\R$ is $\delta_\Gamma(x)$,
Theorem \ref{thm:two-weight} gives
\[
\int_D|V(x,0)|^2\delta_\Gamma^{\lambda-2}\,d\mu
\le CM_{\mathrm g}\int_{D\times\R}w|\nabla V|^2\,dX,
\]
for all $V\in C_c^\infty(D\times\R)$.
The same definition via the closed Dirichlet form applies to unbounded $D$.
The spectral theorem and the scalar profile calculation in
\eqref{eqn:spectral-energy} give the Stinga--Torrea extension $E_su$
\cite[Theorem 1.1]{ST2010}, with
\[
\int_{D\times\R}w|\nabla E_su|^2\,dX
=c_s\|L_D^{s/2}u\|_2^2.
\]
This identity follows first on bounded spectral intervals and then by
Tonelli's theorem. Approximation in energy extends the trace inequality
to $E_su$ and proves \eqref{eqn:spectral-two-weight}.
Taking $d\mu=dx$ and $\lambda=2-2s$ gives the spectral Hardy inequality.
\end{proof}

\section{Nonlinear fractional equations}\label{sec:nonlinear}

Following \cite{CKKSS2025,GS2025}, we define the nonlinear
Dirichlet-to-Neumann operator by the first variation of the minimum
extension energy, or equivalently by the boundary flux of a weighted
$p$-harmonic extension.
For the associated Neumann theory and H\"older estimates with measure data,
see \cite{CGKS2026}.
Here we impose zero exterior values on the trace and apply the weighted
Poisson estimates to obtain solvability under capacity density.

Let $D\subsetneq\R^n$ be nonempty and open, $1<p<\infty$ and $0<s<1$.
We use the slit domain $O$ and weight $w$ in \eqref{eqn:slit-p}.
Assume that $D$ is bounded or that \eqref{eqn:hardy-fractional-p}
holds with $\Gamma=\complement D$, and let $\mathcal X$ be the homogeneous
trace space defined in Section \ref{sec:embeddings}.

Consider the completion of smooth functions compactly supported in the
closed upper half-space, with trace in $C_c^\infty(D)$, in the norm
\[
\left(\int_{\R^n\times(0,\infty)}w|\nabla V|^p\,dX\right)^{1/p}.
\]
Equations \eqref{eqn:extension-p} and \eqref{eqn:homogeneous-trace}
identify its trace space with $\mathcal X$.
Local trace control and the weighted Poincar\'e inequality identify its
elements with local Sobolev functions and fix their additive constants.
It is a closed subspace of the $L^p(w)$ space of gradients and is reflexive.
For $u\in\mathcal X$, let $H_u$ be the energy minimizer with trace $u$.
Weak compactness gives existence, and strict convexity gives uniqueness.
It satisfies
\begin{equation}\label{eqn:nonlinear-extension-harmonic}
-\divergence(w|\nabla H_u|^{p-2}\nabla H_u)=0
\quad\text{in }\R^n\times(0,\infty),\qquad
\trace H_u=u\quad\text{on }\R^n.
\end{equation}
For $\varphi\in\mathcal X$, choose an energy extension $\Phi$ with
trace $\varphi$. The Dirichlet-to-Neumann operator is defined by
\begin{equation}\label{eqn:flux-definition}
\langle\mathcal N_{s,p}^Du,\varphi\rangle
=2\int_{\R^n\times(0,\infty)}
w|\nabla H_u|^{p-2}\nabla H_u\cdot\nabla\Phi\,dX.
\end{equation}
The weak equation \eqref{eqn:nonlinear-extension-harmonic} for zero-trace
variations makes \eqref{eqn:flux-definition} independent of the choice
of $\Phi$. It is the first variation of the minimum extension energy
$2\int_{\R^n\times(0,\infty)}w|\nabla H_u|^p\,dX/p$.
Consequently, for each $\mu\in\mathcal X^*$, minimization gives a unique
$u\in\mathcal X$ satisfying $\mathcal N_{s,p}^Du=\mu$.
For all $\eta\in C_c^\infty(O)$, the even extension of $H_u$ satisfies
\[
\int_O w|\nabla H_u|^{p-2}\nabla H_u\cdot\nabla\eta\,dX
=\langle\mu,\eta(\cdot,0)\rangle.
\]
When $\mu$ is a measure, this is
$-\laplacian_{p,w}H_u=\mu\otimes\delta_0$ in $O$.
For $p=2$, \eqref{eqn:extension} gives
$\mathcal N_{s,2}^Du=c_s(-\laplacian)^su$ in $D$.
For $p\ne2$, \eqref{eqn:flux-definition} defines a separate operator.
The representation in \cite[Theorem 3.1]{delTesoGomezCastroVazquez}
extends $|u(x_0)-u(\cdot)|^{p-2}(u(x_0)-u(\cdot))$ for each $x_0$;
it does not identify the integral fractional $p$-Laplacian with
\eqref{eqn:flux-definition}.

A locally finite measure need not belong to $\mathcal X^*$.
We therefore extend the formulation to the local Sobolev class,
without requiring finite total energy.
\begin{definition}\label{def:flux-local}
Let $\mu$ be a signed locally finite measure on $D$.
A solution of $\mathcal N_{s,p}^Du=\mu$ in the local Sobolev class consists of
a function $u\in C(\R^n)$, zero on $\complement D$, together with a bounded even function
$U\in H^{1,p}_{\loc}(O;w)\cap C(\cl O)$ such that
\[
U(x,0)=u(x)\quad(x\in D),\qquad
U=0\quad\text{on }\complement D\times\{0\},
\]
and
\begin{equation}\label{eqn:flux-local}
\int_Ow|\nabla U|^{p-2}\nabla U\cdot\nabla\eta\,dX
=\int_D\eta(x,0)\,d\mu(x),
\end{equation}
for all $\eta\in C_c^\infty(O)$.
The extension is part of the solution.
If it belongs to the energy completion above, this definition agrees
with \eqref{eqn:flux-definition}.
\end{definition}

The pointwise estimate in Theorem \ref{thm:flux-morrey} below also holds
when capacity density is imposed only on $\Gamma\times\{0\}$, where
$\varnothing\ne\Gamma\subset\partial D$ is closed.
This follows from Proposition \ref{prop:barrier}, with $\delta$ replaced
by $\delta_\Gamma$ in \eqref{eqn:flux-morrey}, the finite-depth assumption
and \eqref{eqn:flux-bound}. The measure in \eqref{eqn:flux-morrey} is
evaluated on $B(x,r)\cap D$.

\begin{theorem}\label{thm:flux-morrey}
Suppose that $O$ satisfies weighted capacity density on $\partial O$.
Let $0<\beta\le ps/(p-1)$ and let $\mu$ be a signed locally finite
measure on $D$ satisfying
\begin{equation}\label{eqn:flux-morrey}
|\mu|(B(x,r))\le Mr^{n-ps+(p-1)\beta}
\quad(x\in D,\ 0<r<\delta(x)/2).
\end{equation}
If $\|\delta\|_{L^\infty(\spt\mu)}<\infty$, there exists a solution
in the sense of Definition \ref{def:flux-local}, and
\begin{align}
|u(x)|&\le CM^{1/(p-1)}\|\delta\|_{L^\infty(\spt\mu)}^{\beta-\alpha_{\mathrm{ss}}}
\min\{\delta(x),\|\delta\|_{L^\infty(\spt\mu)}\}^{\alpha_{\mathrm{ss}}}
,\label{eqn:flux-bound}\\
[u]_{C^{\alpha_{\mathrm{gh}}}(\R^n)}
&\le CM^{1/(p-1)}\|\delta\|_{L^\infty(\spt\mu)}^{\beta-\alpha_{\mathrm{gh}}},
\label{eqn:flux-holder}
\end{align}
for all $x\in D$.
The exponents are those in Theorem \ref{thm:local} on $O$;
they and $C$ depend only on $n,p,s,\beta$ and the capacity density constant.
For $\mu\ge0$, the extension can be chosen minimal among nonnegative solutions.
If $D$ is bounded or $p=2$, both the solution and its bounded extension
are unique. Condition \eqref{eqn:general-fractional-density} with
$\Gamma=\complement D$ is sufficient for the geometric hypothesis.
\end{theorem}
\begin{proof}
Equation \eqref{eqn:flux-morrey} is \eqref{eqn:plane-morrey-p} with
$\lambda=p-(p-1)\beta$. The assumed range of $\beta$ gives
$0\le\lambda<p$.
By \eqref{eqn:lifted-morrey-p}, $\mu\otimes\delta_0$ satisfies
\eqref{eqn:morrey} on $O$.
Moreover, its support lies on the plane, where $\delta_O(x,0)=\delta(x)$.
Theorem \ref{thm:local} therefore gives a bounded solution of
\begin{equation}\label{eqn:flux-extension-problem}
\begin{cases}
-\laplacian_{p,w}U=\mu\otimes\delta_0&\text{in }O,\\
U=0&\text{on }\partial O.
\end{cases}
\end{equation}
In the construction of that theorem, the truncated domains and measures
can be chosen symmetric under $y\mapsto-y$.
Uniqueness of each energy approximation implies evenness of $U$.
Equation \eqref{eqn:flux-extension-problem} then shows that its trace satisfies Definition \ref{def:flux-local}.
Restricting \eqref{eqn:main-bound} and \eqref{eqn:holder} to the plane gives
\eqref{eqn:flux-bound} and the H\"older estimate on $\cl D$.
For $x\in D$ and $z\notin D$, $\delta(x)\le|x-z|$ and $u(z)=0$;
\eqref{eqn:flux-bound} then gives \eqref{eqn:flux-holder} across the boundary.
Minimality for nonnegative data follows from Theorem \ref{thm:local}.
The capacity implication follows from \eqref{eqn:fractional-capacity-lift}.
For $p=2$, uniqueness follows from Theorem \ref{thm:local-uniqueness} on $O$.
For bounded $D$ and general $p$, it follows from Lemma
\ref{lem:decay-infinity} and Proposition \ref{prop:flux-uniqueness} below.
\end{proof}

\begin{lemma}\label{lem:decay-infinity}
Let $D$ be bounded.
Suppose that $U\in H^{1,p}_{\loc}(O;w)\cap C(\cl O)$ is bounded and even,
vanishes on $\complement D\times\{0\}$, and satisfies
\[
-\laplacian_{p,w}U=0\quad\text{in }\R^{n+1}\setminus(\R^n\times\{0\}).
\]
Then
\begin{equation}\label{eqn:flux-decay}
\lim_{R\to\infty}\sup_{X\in O,\ |X|\ge R}|U(X)|=0.
\end{equation}
\end{lemma}
\begin{proof}
Choose $R_0>0$ such that $\cl D\times\{0\}\subset B(0,R_0)$.
Outside this ball, define $H$ by odd reflection of the upper restriction
of $U$ across $y=0$.
We first verify the equation across the plane.
On a compact flat boundary patch outside the ball, the tests
$\eta^p(U-\epsilon)_+$ and $-\eta^p(-U-\epsilon)_+$ vanish near the plane.
The Caccioppoli estimate and $\epsilon\downarrow0$ give
\[
\int_{\R^n\times(0,\infty)}\eta^p|\nabla U|^pw\,dX
\le C\int_{\R^n\times(0,\infty)}|U|^p|\nabla\eta|^pw\,dX.
\]
The signed truncations converge in the local Sobolev norm and have
zero trace. Hence the reflection has local finite energy.
For a smooth test $\varphi$ outside $\cl{B(0,R_0)}$, the reflected
weak equation is the upper-half-space equation tested with
$\varphi(x,y)-\varphi(x,-y)$.
The zero-trace approximation following \eqref{eqn:zero-trace-hardy-p}
justifies this test and proves
\begin{equation}\label{eqn:odd-exterior-equation}
-\laplacian_{p,w}H=0\quad\text{in }\R^{n+1}\setminus\cl{B(0,R_0)}.
\end{equation}

For $r\ge R>2R_0$, a finite chain of balls in
$\{r/2<|X|<2r\}$ joins a point of $\partial B(0,r)$ to its reflection.
The number of balls and the Harnack constants are independent of $r$,
because the weight is homogeneous.
Apply the Harnack inequality to the positive solution
$\sup_{|Z|\ge R/2}|H(Z)|+\epsilon-H$ of
\eqref{eqn:odd-exterior-equation} on that annulus.
Oddness gives, for $|X|=r$,
\[
\sup_{|Z|\ge R/2}|H(Z)|+\epsilon+H(X)
\le C\left(\sup_{|Z|\ge R/2}|H(Z)|+\epsilon-H(X)\right).
\]
Interchanging $X$ with its reflection and letting $\epsilon\downarrow0$
shows that, for a fixed $0<\theta<1$ and every $R>2R_0$,
\[
\sup_{|X|\ge R}|H(X)|
\le\theta\sup_{|X|\ge R/2}|H(X)|.
\]
Iteration makes the left-hand side tend to zero.
Evenness of $U$ transfers this conclusion to \eqref{eqn:flux-decay}.
\end{proof}

\begin{proposition}\label{prop:flux-uniqueness}
Let $D$ be bounded. Two extensions in Definition \ref{def:flux-local}
with the same signed locally finite datum are equal.
A bounded continuous weighted $p$-harmonic function in the upper
half-space is also uniquely determined by a continuous trace supported
in $\cl D$.
\end{proposition}
\begin{proof}
Let $U$ and $V$ satisfy \eqref{eqn:flux-local} for the same datum.
By \eqref{eqn:flux-decay} and their common zero boundary values,
$(U-V-\epsilon)_+$ has compact support in $O$ for each $\epsilon>0$.
Subtracting \eqref{eqn:flux-local} and using this test gives
\[
\int_{O\cap\{U-V>\epsilon\}}w
\bigl(|\nabla U|^{p-2}\nabla U-|\nabla V|^{p-2}\nabla V\bigr)
\cdot(\nabla U-\nabla V)\,dX=0.
\]
Strict monotonicity and the Poincar\'e inequality imply
$(U-V-\epsilon)_+=0$.
Letting $\epsilon\downarrow0$ and interchanging $U$ and $V$ gives
$U=V$ in $O$.
For two upper-half-space extensions of the same compactly supported
trace, the exterior reflection argument in Lemma \ref{lem:decay-infinity}
gives decay at infinity for each extension.
The same comparison in the upper half-space, using the common trace,
proves the second assertion.
\end{proof}

The same reflection argument gives uniqueness for the local Poisson problem
in a half-space when the datum has bounded support.
\begin{corollary}\label{cor:half-space-uniqueness}
Let $n\ge2$, $1<p<\infty$, and let $\nu$ be a signed locally finite
Radon measure on $\R^n_+=\{x_n>0\}$ with bounded support.
Any bounded weak solution
$u\in W^{1,p}_{\loc}(\R^n_+)\cap C(\cl{\R^n_+})$ of
\[
-\laplacian_pu=\nu\quad\text{in }\R^n_+,
\qquad u=0\quad\text{on }\partial\R^n_+
\]
satisfies
\[
\lim_{R\to\infty}\sup_{x\in\R^n_+,\ |x|\ge R}|u(x)|=0.
\]
In particular, there is at most one such solution.
\end{corollary}
\begin{proof}
Choose $R_0>0$ such that $\spt\nu\subset B(0,R_0)$.
Then $u$ is $p$-harmonic in $\R^n_+\setminus\cl{B(0,R_0)}$.
The odd-reflection and annular Harnack argument in the proof of
Lemma \ref{lem:decay-infinity}, with unit weight in dimension $n$,
gives the asserted decay. For two such solutions $u$ and $v$,
$(u-v-\epsilon)_+$ has compact support in $\R^n_+$ for every
$\epsilon>0$, by this decay and their common zero boundary values.
Subtracting the weak equations and testing with this function,
strict monotonicity gives $(u-v-\epsilon)_+=0$.
Letting $\epsilon\downarrow0$ and interchanging $u$ and $v$ proves uniqueness.
\end{proof}

\section*{Acknowledgments}
This work was supported by JSPS KAKENHI Grant Number 21K03281.
This work was supported by the Research Institute for Mathematical
Sciences, an International Joint Usage/Research Center located in Kyoto
University. The
author used ChatGPT (OpenAI) to assist with mathematical discussions, manuscript
preparation, and English-language editing.

\bibliographystyle{amsalpha}
\bibliography{reference}
\end{document}